\documentclass[12pt,a4paper,psamsfonts]{amsart}
\usepackage{fancyhdr}
\usepackage{appendix}
\usepackage{amssymb,amscd,amsxtra,calc}
\usepackage{mathrsfs}
\usepackage{cmmib57}
\usepackage{multirow}
\usepackage[all]{xy}
\usepackage[colorlinks,linkcolor=blue,anchorcolor=blue,citecolor=green]{hyperref}
\theoremstyle{plain}
    \newtheorem{thm}{Theorem}[section]
    
    \newtheorem{claim}[thm]{Claim}
     \newtheorem{defi}[thm]{Definition}
    \newtheorem{coro}[thm]{Corollary}
    
    \newtheorem{lem}[thm]{Lemma}
    \newtheorem{pro}[thm]{Proposition}

    \newtheorem{remark}[thm]{Remark}

\makeatletter

\newcommand{\Rmnum}[1]{\expandafter\@slowromancap\romannumeral #1@}
\makeatother
\begin{document}

\bibliographystyle{plain}
\title[Totally invariant divisors of int-amplified endomorphisms]
{Totally invariant divisors of int-amplified endomorphisms of normal projective varieties
}

\author{Guolei Zhong}
\address
{
\textsc{Department of Mathematics} \endgraf
\textsc{National University of Singapore,
Singapore 119076, Republic of Singapore
}}
\email{zhongguolei@u.nus.edu}

\begin{abstract}
We consider an arbitrary int-amplified surjective endomorphism $f$ of a normal projective variety $X$ over $\mathbb{C}$ and its $f^{-1}$-stable prime divisors. We  extend the early result in  \cite[Theorem 1.3]{zhang2014invariant} for the case of polarized endomorphisms to the case of int-amplified endomorphisms.

Assume further that $X$ has at worst Kawamata log terminal singularities. We prove that the total number of  $f^{-1}$-stable prime divisors has an optimal upper bound $\dim X+\rho(X)$, where $\rho(X)$ is the Picard number. Also, we  give a sufficient condition for $X$ to be rationally connected and simply connected. Finally, by running the minimal model program (MMP), we prove that, under some extra conditions, the end product of the MMP can only be an elliptic curve or a single point.
\end{abstract}


\maketitle

\tableofcontents

\section{Introduction}
~~~~We work over the complex numbers field $\mathbb{C}$. This article  generalizes the early result of \cite[Theorem 1.3]{zhang2014invariant}. We  extend the result for the case of polarized endomorphisms to the case of int-amplified endomorphisms. The main method we use is to run the minimal model program (MMP) equivariantly so that we can reduce the dimension of varieties and then use the induction. It is accessible under some good conditions for the variety $X$ (cf.~\cite[Theorem 1.10]{meng2017building}). See Theorem \ref{thm1.10} for a detailed description.

Suppose $X$ is a projective variety. Let $f: X\rightarrow X$ be a surjective endomorphism. The endomorphism $f$ is said to be \textit{int-amplified}, if there exist ample Cartier divisors $H$ and $L$ such that $f^*L-L=H$. A prime divisor $V$ on $X$ is said to be \textit{totally invariant} under the endomorphism $f$, if $f^{-1}(V)=V$ set-theoretically. To simplify the expression, we first introduce a notation for the set of totally invariant prime divisors for the int-amplified endomorphism $f:X\rightarrow X$:
$$\textup{TI}_f(X)=\{V~|~V~\textup{is a prime divisor on }X~\textup{such that }f^{-1}(V)=V\}.$$

In this article, we will bound the cardinality $c:=\#\textup{TI}_f(X)$. It turns out that there exists an optimal upper bound which is determined by the Picard number and dimension of $X$. Also, we want to give a sufficient condition for  $f^*|_{\textup{Pic}(X)}$ to be diagonalizable over $\mathbb{Q}$ by applying some early results. Moreover, under some extra conditions, we show that the variety $X$ is rationally connected and simply connected with respect to complex topology.

The following is our main result.

\begin{thm}\label{thm1.1}
Let $X$ be a projective variety of dimension $n$ with only $\mathbb{Q}$-factorial Kawamata log terminal singularities, and $f:X\rightarrow X$ an int-amplified endomorphism. Let $V_j~(1\le j\le c)$ be all the prime divisors in $\textup{TI}_f(X)$. Then we have \textup{(}with $\rho:=\rho(X)$\textup{)}:
\begin{enumerate}
\item[\textup{(1)}] $c\le n+\rho$.  Furthermore, if $c\ge 1$, then the pair $(X,\sum V_j)$ is log canonical and $X$ is uniruled.
\item[\textup{(2)}] Suppose  $c\ge n+\rho-2$. Then either $X$ is rationally connected and simply connected with respect to complex topology, or there is a fibration $X\rightarrow E$ onto an elliptic curve $E$ such that every fibre is normal, irreducible, equi-dimensional and rationally connected. Further, in the latter case, for some integer $t\ge 1$, $f^t$ descends to an int-amplified endomorphism $g_E:E\rightarrow E$. In both cases, $(f^t)^*|_{\textup{NS}_{\mathbb{Q}}(X)}$ is diagonalizable over $\mathbb{Q}$.
\item[\textup{(3)}] Suppose  $c\ge n+\rho-1$. Then $X$ is rationally connected and simply connected with respect to complex topology. Further, for some integer $t\ge 1$, $(f^t)^*|_{\textup{Pic}(X)}$ is diagonalizable over $\mathbb{Q}$.
\item[\textup{(4)}] Suppose  $c\ge n+\rho$. Then $c=n+\rho$, $K_X+\sum_{j=1}^{n+\rho}V_j\sim_{\mathbb{Q}}0$ and $f$ is \'etale outside $(\bigcup V_j)\cup f^{-1}(\textup{Sing}\,X)$.
\end{enumerate}
\end{thm}

One direct result of our main theorem is as follows. We refer to \cite{brown2016geometric} for the toric pair.
\begin{coro}\label{coro1.2}
Suppose $X$ is a normal projective variety with only $\mathbb{Q}$-factorial Kawamata log terminal singularities, and $f:X\rightarrow X$ is an int-amplified endomorphism of $X$. If the cardinality $c$ of $\textup{TI}_f(X)$ satisfies $c=\rho(X)+\dim X$ (achieving the upper bound) as in Theorem \ref{thm1.1} (4), then $(X,\sum V_j)$ is a toric pair.
\end{coro}
Indeed, Corollary \ref{coro1.2} follows immediately from \cite[Theorem 1.2]{brown2016geometric} and our main theorem, in which case, the complexity \textup{(}cf.~\cite[Definition 1.1]{brown2016geometric}\textup{)} is zero. So $(X,\sum V_j)$ is a toric pair.

Let $\textup{N}^1(X)=\textup{NS}(X)\otimes_{\mathbb{Z}}\mathbb{R}$. Comparing with the case of polarized endomorphisms, we remove the hypothesis (cf.~\cite[Theorem 1.3]{zhang2014invariant}): either $f^*|_{\textup{N}^1(X)}=q~\textup{id}_{\textup{N}^1(X)}$, or $n\le 3$. In general, $f^*|_{\textup{N}^1(X)}$ may not be a scalar matrix even in the case of polarized endomorphisms (cf.~\cite[Example 7.1]{meng2018building}). In our case of int-amplified endomorphisms, the diagonalizable result holds even when $c=\rho(X)+\dim X-1$, extending \cite[Theorem 1.3]{zhang2014invariant}. Moreover, we add more details to the proof of \cite[Theorem 1.3 and Proposition 2.12]{zhang2014invariant}. 

For the organization of the paper, we begin with several preliminaries in Section \ref{section2}. Then we give a detailed proof for our Theorem \ref{thm1.1} in Section \ref{section3}.
\subsubsection*{\textbf{\textup{Acknowledgments}}}
The author would like to deeply thank Professor De-Qi Zhang for many inspiring ideas and discussions. Also, he would like to thank Doctor Sheng Meng for the result of running MMP  \cite{meng2017building} on int-amplified endomorphisms, and the referees for many constructive suggestions to improve the paper.

\section{Preliminaries}\label{section2}
\subsection{Notation and terminology}\label{section2.1}

Let $X$ be a $\mathbb{Q}$-factorial normal projective variety of dimension $n$ over $\mathbb{C}$. Let $D$ be a Cartier divisor on $X$, and $\mathcal{O}_X(D)$ its corresponding invertible sheaf. We often identify $D$ with $\mathcal{O}_X(D)$. Since $X$ is $\mathbb{Q}$-factorial, any Weil divisor $B$ on $X$ is $\mathbb{Q}$-Cartier, 
i.e.\,there exists some integer $m>0$, such that $mB$ is Cartier. We use $K_X$ to denote a canonical divisor of the variety $X$.

In correspondence with notations in \cite{hartshorne2013algebraic}, let $\textup{NS}(X)=\textup{Pic}(X)/\textup{Pic}^\circ(X)$ denote the N\'eron--Severi group of $X$. Let $\textup{NS}_{\mathbb{Q}}(X):=\textup{NS}(X)\otimes_{\mathbb{Z}}\mathbb{Q}$ and $\textup{N}^1(X):=\textup{NS}(X)\otimes_{\mathbb{Z}}\mathbb{R}$. Denote by $q(X)$  the irregularity $h^1(X,\mathcal{O}_X):=\dim \textup{H}^1(X,\mathcal{O}_X)$. 

Let $\textup{N}_{n-1}(X)$ denote the space of weakly numerically equivalent classes of Weil $\mathbb{R}$-divisors (cf.~\cite[Definition 2.2]{meng2018building}). When $X$ is normal, we can regard $\textup{N}^1(X)$ as a subspace of $\textup{N}_{n-1}(X)$ (cf.~\cite[Lemma 3.2]{zhang2016action}). 
    
    Let $f:X\rightarrow X$ be a finite surjective endomorphism. We can define the pullback of $r$-cycles for $f$, such that $f^*$ induces an automorphism of $\textup{N}_r(X)$ and $f_*f^*=(\deg f)~\textup{id}$. More precisely, the pullback $f^*:\textup{N}_r(X)\rightarrow \textup{N}_r(X)$ is defined by $f^*=(\deg f)~(f_*)^{-1}$. We recall the following  cones, which are $(f^*)^{\pm}$-invariant (cf.~\cite[Definition 2.4]{meng2018building}). 
    \begin{itemize}
\item $\textup{Amp}(X)$: the set of classes of ample $\mathbb{R}$-Cartier divisors in $\textup{N}^1(X)$;
\item $\textup{Nef}(X)$: the set of classes of nef $\mathbb{R}$-Cartier divisors in $\textup{N}^1(X)$;
\item $\textup{PE}(X)$: the closure of the set of classes of effective $(n-1)$-cycles with $\mathbb{R}$-coefficients in $\textup{N}_{n-1}(X)$.
\end{itemize}

The Theorem of the Base of N\'eron--Severi asserts that the real space $\textup{N}_1(X)$ of $1$-cycles with real coefficients modulo numerical equivalence is a finite dimensional $\mathbb{R}$-vector space. There is a natural perfect pairing:
$$\textup{N}_1(X)\times \textup{N}^1(X)\rightarrow \mathbb{R}.$$
Then $\textup{rank}_{\mathbb{R}}\textup{N}^1(X)=\textup{rank}_{\mathbb{R}}\textup{N}_1(X)=\rho(X)$, the \textit{Picard number} of $X$.

We use the symbol $\sim$ (resp.~$\sim_{\mathbb{Q}}$, $\equiv$ or $\equiv_w$) to denote the linear (resp.~$\mathbb{Q}$-linear, numerical or weak numerical) equivalence relation. By \cite[CH.~\Rmnum{13}, Theorem 4.6]{berthelot2006theorie}), we see that numerically equivalent divisors are algebraically equivalent up to a positive multiple.

We refer to \cite{nakayama2004zariski} for the definition and general properties of a pseudo-effective Weil divisor on $X$. Also, we refer to \cite[CH.2]{lazarsfeld2017positivity} for the definition and general properties of Kodaira dimension and Iitaka dimension of any line bundle over $X$. 

In addition, we now assume that $X$ has at worst Kawamata log terminal (klt) singularities. We refer to \cite[Definitions 2.28 and 2.34]{kollar2008birational} for the definition of discrepancies and different kinds of singularities. The next proposition about the inversion of adjunction (cf.~\cite[Theorem 1]{kawakita2007inversion}) is useful 
for our proof.
\begin{pro}[cf.~\cite{kawakita2007inversion}]\label{normallc}
Let $(X,S+B)$ be a log pair \textup{(}i.e.\,$K_X+S+B$ is $\mathbb{Q}$-Cartier and all the coefficients of irreducible components of $S+B$ are in $[0,1]$\textup{)} such that $S$ is a reduced divisor which has no common component with the support of $B$, let $\nu:S^{\nu}\rightarrow X$ denote the normalization of $S$, and let $B^{\nu}$ denote the different of $B$ on $S^{\nu}$ \textup{(}so that $K_{S^{\nu}}+B^{\nu}=\nu^*(K_X+S+B)$\textup{)}. Then $(X,S+B)$ is log canonical near $S$ if and only if $(S^{\nu},B^{\nu})$ is log canonical.
\end{pro}

Suppose $f: X\rightarrow X$ is a finite surjective endomorphism of a normal variety $X$. Since $X$ is normal, it is regular in codimension one. Then one can define the pullback of a Weil divisor $H$ on $X$, as the closure of $(f|_U)^*(H|_U)$ where $U\subseteq X$ is a smooth open locus on $X$ and $X\backslash U$ is a codimension $\ge 2$ closed subset. Furthermore, when $H$ is $\mathbb{Q}$-Cartier, the pullback we discussed above coincides with the usual pullback of $\mathbb{Q}$-Cartier divisor.
\setlength\parskip{8pt}

\subsection{Int-amplified endomorphisms}
~~~~In this subsection, we first recall the definitions of polarized, amplified and int-amplified endomorphisms. Then we refer to \cite{meng2017building} for the general properties of int-amplified endomorphisms.\setlength\parskip{0pt}
\begin{defi}\label{defiample}
\textup{Let $f: X\rightarrow X$ be a surjective endomorphism of a projective variety $X$. We say that
\begin{enumerate}
\item[\textup{(1)}] $f$ is \textit{polarized} if $f^* D\sim qD$ for some ample Cartier divisor $D$ and integer $q>1$;
\item[\textup{(2)}] $f$ is \textit{amplified} if $f^*D-D=H$ for some Cartier divisor $D$ and ample Cartier divisor $H$; and
\item[\textup{(3)}] $f$ is \textit{int-amplified} if $f^*D-D=H$ for some ample Cartier divisors $D$ and $H$.
\end{enumerate}}
\end{defi}
It follows from the definition that $(1)\Rightarrow (3)\Rightarrow (2)$. One can also check directly that if $f$ is an int-amplified endomorphism, then any power of $f$ is also int-amplified.

In what follows, we will recall the  theorem below (cf.~\cite[Theorem 1.10]{meng2017building}), which ensures that we can run equivariant MMP and then do the induction on the dimension of $X$. It extends the result of equivariant MMP (cf.~\cite[Theorem 1.8]{meng2018building}) for the case of  polarized endomorphisms. Recall that a normal projective variety $X$ is said to be \textit{$Q$-abelian} if there exists a finite surjective morphism $A\rightarrow X$ \'etale in codimension one (or \textit{quasi-\'etale} in short) with
$A$ an abelian variety.
\begin{thm}[cf.~\cite{meng2017building}]\label{thm1.10}
Let $f: X\rightarrow X$ be an int-amplified endomorphism of a $\mathbb{Q}$-factorial Kawamata log terminal projective variety $X$. Then replacing $f$ by a positive power, there exist a Q-abelian variety $Y$, a morphism $X\rightarrow Y$, and an $f$-equivariant relative minimal model program over $Y$
\[
\xymatrix{
X=X_0\ar@{-->}[r]&X_1\ar@{-->}[r]&\cdots\ar@{-->}[r]&X_i\ar@{-->}[r]&\cdots\ar@{-->}[r]&X_l\ar@{-->}[r]&X_{l+1}=Y
}
\]
which means a positive power of $f$ descends to an endomorphism $g_i$ on each $X_i$, for $1\le i\le l+1$, with every $u_i:X_i\dashrightarrow X_{i+1}$ a divisorial contraction, a flip or a Fano contraction over $Y$, of a $K_{X_i}$-negative extremal ray. Further, we have:
\begin{enumerate}
\item[\textup{(1)}] If $K_X$ is pseudo-effective, then $X=Y$ and it is Q-abelian.
\item[\textup{(2)}] If $K_X$ is not pseudo-effective, then for each $i$, $X_i\rightarrow Y$ is equi-dimensional and holomorphic with every fibre \textup{(}irreducible\textup{)} rationally connected and $g_i$ is int-amplified. The last rational map $X_l\dashrightarrow X_{l+1}=Y$ is a Fano contraction \textup{(}a morphism\textup{)}.
\item[\textup{(3)}] $f^*|_{\textup{N}^1(X)}$ is diagonalizable over $\mathbb{C}$ if and only if so is $g_{l+1}^*|_{\textup{N}^1(Y)}$.
\end{enumerate}
\end{thm}
To make our symbols symmetric, we may denote by $g_0$ some power $f^t$, so that $g_0$ descends to the int-amplified endomorphism $g_i$ of $X_i$ for each $i$. 
\begin{remark}
\textup{With the same symbols given above, we shall prove later that, for a divisorial contraction or a flip $X_i\dashrightarrow X_{i+1}$, we have the equality $\#\textup{TI}_{g_{i+1}}(X_{i+1})=\#\textup{TI}_{g_i}(X_i)-(\rho(X_i)-\rho(X_{i+1}))$ (cf.~Lemma \ref{invariant}), while for a Mori fibre contraction $X_r\rightarrow  X_{r+1}$, the inequality $\#\textup{TI}_{g_{r+1}}(X_{r+1})\ge\#\textup{TI}_{g_r}(X_r)-(\dim X_r-\dim X_{r+1})-1$ holds (cf.~Lemma \ref{invariant1}). Here, $\#\textup{TI}_{g_i}(X_i)$ denotes the cardinality of the set $\textup{TI}_{g_i}(X_i)$ of totally invariant divisors for the int-amplified endomorphism $g_i:X_i\rightarrow X_i$. }	
\end{remark}

When proving Theorem \ref{thm1.10}, the following statements for the case of int-amplified endomorphisms are useful. Lemma \ref{lem3.6} is a special case of  \cite[Lemma 3.5]{meng2017building} and Lemma \ref{prop3.3} was proved in \cite[Theorem 3.3]{meng2017building}.

\begin{lem}[cf.~\cite{meng2017building}]\label{lem3.6}
Let $\pi: X\rightarrow Y$ be a generically finite and surjective morphism of projective varieties. Suppose $f:X\rightarrow X$ and $g:Y\rightarrow Y$ are two surjective endomorphisms such that $g\circ \pi=\pi\circ f$. Then $f$ is int-amplified if and only if so is $g$.
\end{lem}
\begin{lem}[cf.~\cite{meng2017building}]\label{prop3.3}
Let $f: X\rightarrow X$ be an int-amplified surjective endomorphism of a projective variety $X$. Then all the eigenvalues of $f^*|_{\textup{N}^1(X)}$ are of modulus greater than $1$.
\end{lem}
If $X$ admits an int-amplified endomorphism $f$, then for any $D\in\textup{TI}_f(X)$, $f^*D=sD$ for some integer $s>1$. Further, we have $\deg f>1$ (cf.~\cite[Lemma 3.7]{meng2017building}).

An amplified morphism was first defined by Krieger and Reschke (cf.~\cite{krieger2015cohomological}), and Fakhruddin showed the following very motivating result in  \cite[Theorem 5.1]{fakhruddin2002questions}. A subset $Z\subseteq X$ is said to be \textit{$f$-periodic} if $f^s(Z)=Z$ for some $s>0$. Lemma \ref{densepoint} says that $f$-periodic points are dense if $f$ is amplified.
\begin{lem}[cf.~\cite{fakhruddin2002questions}]\label{densepoint}
Let $X$ be a projective variety over an algebraically closed field $k$. Suppose $f: X\rightarrow X$ is a dominant morphism and $\mathcal{L}$ a line bundle on $X$ such that $f^*\mathcal{L}\otimes \mathcal{L}^{-1}$ is ample. Then the subset $X(k)$ consisting of periodic points of $f$ is Zariski dense in $X$.
\end{lem}

\setlength\parskip{8pt}
\subsection{Rational connectedness of varieties}
~~~~We refer to \cite[CH.~\Rmnum{4}, Definitions 1.1 and 3.2]{kollar1999rational} for the definitions of \textit{uniruled varieties} and \textit{rationally connected varieties}. We recall that  a complete variety $X$ is  \textit{log $Q$-Fano}, if there exists an effective $\mathbb{Q}$-divisor $D$ such that the pair $(X,D)$ is klt and $-(K_X+D)$ is an ample $\mathbb{Q}$-Cartier divisor.\setlength\parskip{0pt}

\begin{remark}\label{remncu}
\textup{If $X$ is a $\mathbb{Q}$-Gorenstein normal projective variety over an algebraically closed field of characteristic zero with the canonical divisor $K_X$ not pseudo-effective, then $X$ is uniruled (cf.~\cite[0.3 Corollary]{boucksom2004pseudo}). }
\end{remark}

Next, we review the properties of the relations between uniruled and rationally connected varieties (cf.~\cite[Proposition 3.3]{kollar1999rational}).
\begin{pro}[cf.~\cite{kollar1999rational}]\label{prorc}
Suppose $X$ is a variety over a field $k$.
\begin{enumerate}
\item[\textup{(1)}] If $X$ is rationally connected, then $X$ is uniruled.
\item[\textup{(2)}] Let $X$ and $X'$ be two proper varieties, birational to each other. Then $X$ is rationally connected if and only if so is $X'$.
\end{enumerate}
\end{pro}

Throughout the proof of Theorem \ref{thm1.1}, we also need several topological facts for rationally connected varieties. 
Recall that a path-connected topological space $X$ is \textit{simply connected} if and only if its fundamental group $\pi_1(X)$ is trivial. Besides, \textit{algebraic fundamental group} $\pi_1^{\textup{alg}}(X)$ (cf.~\cite[Definition 3.5.43]{poonen2017rational}) is a profinite completion of the topological fundamental group $\pi_1(X)$ (cf.~\cite[Theorem 3.5.41]{poonen2017rational}). 

Suppose $X$ is a connected variety over a separably closed field. Then comparing with simply connected (with respect to complex topology) varieties, the variety $X$ is said to be \textit{algebraically simply connected} if it has no nontrivial connected finite \'etale cover, which is equivalent to  $\pi_1^{\textup{alg}}(X)$ being trivial (cf.~\cite[Definition 3.5.45]{poonen2017rational}).

Moreover, from Universal Coefficient Theorem for Cohomology (cf.~\cite[Theorem 3.2]{hatcher2002algebraic}) and Hodge theory, we have the following result.
\begin{lem}\label{univ}
Suppose $X$ is a smooth variety over $\mathbb{C}$ with trivial fundamental group. Then the irregularity $q(X):=h^1(X,\mathcal{O}_X)=\dim \textup{H}^1(X,\mathcal{O}_X)=0$.
\end{lem}

We refer to \cite[Theorem 1.1]{takayama2003local} for the following very useful lemma. 
\begin{lem}[cf.~\cite{takayama2003local}]\label{lsco}
Let $X$ be a normal variety and  $f: Y\rightarrow X$  a resolution of singularities. Then the induced homomorphism $f_*: \pi_1(Y)\rightarrow \pi_1(X)$ is an isomorphism if the pair $(X,\Delta)$ is klt for some $\Delta$. 
\end{lem}

 In addition, a well-known result (cf.~\cite[Theorem 3.5]{campana1991twistor} and \cite{kollar1992rational}) gives us when the smooth varieties will be simply connected with respect to complex topology.
\begin{thm}[cf.~\cite{campana1991twistor} and \cite{kollar1992rational}]\label{scc}
A smooth, proper and rationally connected variety is simply connected with respect to complex topology.
\end{thm}
\setlength\parskip{8pt}
\subsection{Properties for polarized cases}
~~~~At the end of this preliminary, we consider a polarized endomorphism $f$ on a normal variety $X$. First, we recall the following result which was proved in \cite[Proposition 2.1]{zhang2014invariant}. \setlength\parskip{0pt}
\begin{lem}[cf.~\cite{zhang2014invariant}]\label{lem1}
Let $X$ be a normal variety, $f:X\rightarrow X$ a surjective endomorphism of $\deg(f)>1$ and $D$ a nonzero reduced divisor with $f^{-1}(D)=D$. Assume:
\begin{enumerate}
\item[\textup{(1)}] $X$ is log canonical around $D$;
\item[\textup{(2)}] $D$ is $\mathbb{Q}$-Cartier; and
\item[\textup{(3)}] $f$ is ramified around $D$.
\end{enumerate}
Then the pair $(X,D)$ is log canonical around $D$. In particular, the reduced divisor $D$ is normal crossing outside the union of $\textup{Sing}\,X$ and a codimension three subset of $X$.
\end{lem}


For a linear map $\phi:V\rightarrow V$ of a finite dimensional real normed vector space $V$, denote by $||\phi||$ the norm of $\phi$.
The following proposition gives us a criterion for $\phi$ to be diagonalizable (cf.~\cite[Definition 2.6, Proposition 2.9]{meng2018building} and \cite[Proposition 3.1]{cascini2017polarized}).
\begin{pro}[cf.~\cite{cascini2017polarized}]\label{mmwan123}
Let $\phi: V\rightarrow V$ be an invertible linear map of a positive dimensional real normed vector space $V$. Assume $\phi(C)=C$ for a convex cone $C\subseteq V$ such that $C$ spans $V$ and its closure $\overline{C}$ contains no line. Let $q$ be a positive number. Then the conditions \textup{(i)} and \textup{(ii)} below are equivalent.
\begin{enumerate}
\item[\textup{(i)}] $\phi(u)=qu$ for some $u\in C^o$ \textup{(}the interior part of $C$\textup{)}.
\item[\textup{(ii)}] There exists a constant $N>0$, such that $\frac{||\phi^i||}{q^i}<N$ for all $i\in\mathbb{Z}$.
\end{enumerate}
Assume further the equivalent conditions (i) and (ii). Then the following are true.
\begin{enumerate}
\item[\textup{(1)}] $\phi$ is a diagonalizable linear map with all eigenvalues of modulus $q$.
\item[\textup{(2)}] Suppose $q>1$. Then for any $v\in V$ such that $\phi(v)-v\in C$, we have $v\in C$.
\end{enumerate}
\end{pro}

\begin{coro}\label{remcariter}
Suppose $f:X\rightarrow X$ is a polarized endomorphism on a normal projective variety $X$ such that $f^*H\equiv qH$ for a positive number $q$ and an ample divisor $H$ on $X$. Then the linear operation $f^*|_{\textup{N}_{n-1}(X)}$ is diagonalizable with all eigenvalues of modulus $q$. In particular, if $D$ is an  effective reduced Weil divisor such that $f^{-1}D=D$, then with $f$ replaced by its power, $f^*D_i=qD_i$ for each irreducible component $D_i$ of $D$ \textup{(}and then $f^*D=qD$\textup{)}.
\end{coro}

\begin{proof}
We may regard $\textup{N}^1(X)$ as a subspace of $\textup{N}_{n-1}(X)$ (cf.~\cite[Lemma 3.2]{zhang2016action}) and consider the invertible linear map  $\phi=f^*|_{\textup{N}_{n-1}(X)}$. Let $V:=\textup{N}_{n-1}(X)$ and $C:=\textup{PE}(X)\subseteq V$ as we defined in Subsection \ref{section2.1}.  Since $H$ is ample, the volume $\textup{vol}(H)>0$ and thus $H$ lies in the interior part of  $\textup{PE}(X)$ if we regard $H$ as a big Weil $\mathbb{R}$-divisor (cf.~\cite[Theorem 3.5 (ii), (iii)]{fulger2016volume} and  \cite[Definition 2.4]{meng2018building}). In addition, $\textup{PE}(X)$ spans the whole $\textup{N}_{n-1}(X)$ and then by Proposition \ref{mmwan123}, $\phi$ is a diagonalizable linear map with all eigenvalues of modulus $q$. 

Since $D$ is $f^{-1}$-invariant, after replacing $f$ by its power, we may assume $f^{-1}D_i=D_i$ for each irreducible component $D_i$ of $D$. Therefore, we can see from the above discussion that $f^*D_i=\phi(D_i)=qD_i$ for each component $D_i$.
\end{proof}
\begin{remark}\label{rem12345}
\textup{
Indeed, the second part of Corollary \ref{remcariter} also follows easily from the projection formula (cf.~\cite[Proposition 2.3]{fulton2013intersection}). Besides, the above corollary will fail if we remove the condition that $f$ is polarized. For example (provided by De-Qi Zhang), consider the product $X=\mathbb{P}^1\times \mathbb{P}^1$ and $f=f_1\times f_2$, where $f_i$ is the power map of $\mathbb{P}^1$, mapping $[x,y]$ to $[x^{q_i},y^{q_i}]$ with $q_1\neq q_2$. Let $D_1=\{a_1\}\times \mathbb{P}^1$, $D_2=\mathbb{P}^1\times \{a_2\}$ and $D=D_1+D_2$, where $a_i$ is one of two $f_i^{-1}$-invariant (coordinate) points. Then $D$ is an ample Cartier divisor. However, by projection formula, $f^*D_i=q_iD_i$. Moreover, in this case, we can get four $f^{-1}$-invariant prime divisors.}
\end{remark}

The following proposition extends \cite[Proposition 2.12]{zhang2014invariant} to the case when $V_j$ are reduced divisors. We first state this generalized result and then do some preparations for its proof.
\begin{pro}\label{promm100}
Let $X$ be a normal projective variety of dimension $n\ge 2$, $V_j~(1\le j\le s)$ reduced divisors \textup{(}may have more than one components\textup{)}, and $f: X\rightarrow X$ a polarized endomorphism with $\deg (f)=q^n$ \textup{(}$q>1$ an integer\textup{)} such that $K_X$ is $\mathbb{Q}$-Cartier and
\begin{enumerate}
\item[\textup{(1)}] $X$ has only log canonical singularities around $\bigcup V_j$;
\item[\textup{(2)}] every $V_j$ is $\mathbb{Q}$-Cartier and ample;
\item[\textup{(3)}] $f^{-1}(V_j)=V_j$ for all $j$; and
\item[\textup{(4)}] $V_i$ and $V_j$ $(i\neq j)$ have no common irreducible components.
\end{enumerate}
Then $s\le n+1$; and $s=n+1$ only if: $f$ is \'etale outside $(\bigcup V_j)\cup f^{-1}(\textup{Sing}\,X)$ and $K_X+\sum_{j=1}^{n+1}V_j\equiv 0$.
\end{pro}

\begin{remark}\label{mm2.21}
\textup{With the same assumption in Proposition \ref{promm100}, the remaining necessary condition for $s=n+1$ in \cite[Proposition 2.12]{zhang2014invariant} still holds: each $V_j$ is irreducible and $\bigcap_{i=1}^t V_{b_i}\subseteq X$ is a normal irreducible subvariety for every subset $\{b_1,\cdots, b_t\}\subseteq \{1,\cdots,n+1\}$ with $1\le t\le n-2$ (cf.~\cite[Claim 2.11]{zhang2014invariant}, Proof of Proposition \ref{promm100} and Remark \ref{mm2.24}).}
\end{remark}

 Before proving Proposition \ref{promm100}, we first prove the following lemmas. Lemma \ref{lemnumer} follows immediately  from Corollary \ref{remcariter}.\begin{lem}\label{lemnumer}
Suppose $f:X\rightarrow X$ is a polarized endomorphism on a normal projective variety $X$ of dimension $n$ and $\deg f=q^n>1$. Then any Weil divisor $M$ \textup{(}not necessarily effective\textup{)} on $X$ such that $f^*M\equiv_w M$, is weakly numerically trivial. In particular, suppose further that $M$ is $\mathbb{Q}$-Cartier. Then $M\equiv 0$.
\end{lem}
\begin{proof}
Since $f$ is polarized by an ample divisor $H$ on $X$, the linear operation $f^*|_{\textup{N}_{n-1}(X)}$ is diagonalizable with all eigenvalues of modulus $q>1$ (cf.~Corollary \ref{remcariter}). Suppose $M\not\equiv_w 0$. Then $1$ is an eigenvalue of the linear operation $f^*|_{\textup{N}_{n-1}(X)}$, a contradiction. Therefore, $M\equiv_w0$. Further, if $M$ is $\mathbb{Q}$-Cartier, then $M\equiv 0$ (cf.~\cite[Lemma 3.2]{zhang2016action}).
\end{proof}
We follow the idea of \cite[Lemma 2.7]{zhang2014invariant} to prove the following result. 
\begin{lem}\label{mmnxbxhwa}
Suppose $f:X\rightarrow X$ is a finite surjective endomorphism on a normal projective variety $X$ of dimension $n\ge 1$. Suppose further that $M$ is a pseudo-effective Weil divisor and $E$ is an effective $\mathbb{Q}$-divisor such that the following weakly numerical equivalence $(*)$ holds:                                                                                                                                                                                                                                                                                                                                                                                                                                                                                                                                                                                                                                                                                                                                                                                                                                                                                                                                                                                                                                                                                                                                                                                                                                                                                                                                                                                                                                                                                                                                                                                                                                                                                                                                                                                                                                                                                                                                                                                                                                                                                                                                                                                                                                                                                                                                                                                                                                                                                                                                                                                                                                                                                                                                                                                                                                                                                                                                                                                                                                                                                                                                                                                                                                                                                                                                                                                                                                                                                                                                                                                                                                                                                                                                                                                                                                                                                                                                                                                                                                                                                                                                                                                                                                                                                                                                                                                                                                                                                                                                                                                                                                                                                                                                                                                                                                                                                                                                                                                                                                                                                                                                                                                                                                                                                                                                                                                                                                                                                                                                                                                                                                                                                                                                                                                                                                                                                                                                                                                                                                                                                                                                                                                                                                                                                                                                                                                                                                                                                                                                                                                                                                                                                                                                                                                                                                                                                                                                                                                                                                                                                                                                                                                                                                                                                                                                                                                                                                                                                                                                                                                                                                                                                                                                                                                                                                                                                                                                                                                                                                                                                                                                                                                                                                                                                                                                                                                                                                                                                                                                                                                                                                                                                                                                                                                                                                                                                                                                                                                                                                                                                                                                                                                                                                                                                                                                                                                                                                                                                                                                                                                                                                                                                                                                                                                                                                                                                                                                                                                                                                                                                                                                                                                                                                                                                                                                                                                                                                                                                                                                                                                                                                                                                                                                                                                                                                                                                                                                                                                                                                                                                                                                                                                                                                                                                                                                                                                                                                                                                                                     
$$(*)~M\equiv_w f^*M+ E.$$
Then the effective $\mathbb{Q}$-divisor $E=0$.
\end{lem}
\begin{proof}
Suppose $E>0$. Multiplying $(*)$ by a positive integer, we may assume $E$ is integral. Substituting the above expression of $M$ to the right-hand side $(k-1)$-times, we get
$$M\equiv_w (f^k)^*M+\sum\limits_{i=0}^{k-1}(f^i)^*E.$$
Taking a fixed ample Cartier divisor $H$ on $X$ and then using Nakai--Moishezon criterion (cf.~\cite[Theorem 1.37]{kollar2008birational}), we have the following 
$$(M\cdot H^{n-1})\ge \sum\limits_{i=0}^{k-1}((f^i)^*E\cdot H^{n-1}).$$
Since $E$ is integral, the right-hand side tends to infinity if we let $k\rightarrow\infty$, a contradiction. Hence, $E=0$.
\end{proof}

Now, we begin with the proof of Proposition \ref{promm100}. We follow the steps and use the similar method given in \cite[the proof of Proposition 2.12]{zhang2014invariant}. Besides, readers may refer to \cite[Lemma 2.8]{zhang2014invariant} for a further proof of Remark \ref{mm2.21}.
\begin{proof}
Suppose $s\ge n+1$. If $n=2$, then we may go to the end product $X_{n-2}$ with $n=2$. Therefore, we may further assume that initially $n\ge 3$ and $f^*H\sim qH$ for an ample divisor $H$ on $X$. Since the number of $V_j$ and the irreducible components of $V_j$ are finite, with $f$ replaced by its power, we may assume for each irreducible component $B_{ij}$ of $V_j$,  $f^{-1}(B_{ij})=B_{ij}$. By Corollary \ref{remcariter}, $f^*B_{ij}=qB_{ij}$. Note that Proposition \ref{promm100} still holds with $f$ replaced by its power since $f$ is \'etale away from $\cup V_j$ if so is its power.

We reduce the dimension of $X$ by continuously taking normalization of divisors. Then we prove Proposition \ref{promm100} by an early result for the surface case. Let $D_0=\sum V_j$ and consider the log ramification divisor formula for the pair $(X,D_0)$ (cf.~\cite[Theorem 11.5]{iitaka1982algebraic}):
\begin{equation}\label{mm0}
K_X+D_0=f^*(K_X+D_0)+\Delta_f,
\end{equation}
where $\Delta_f$ is an effective (integral) divisor, having no common components with $D_0$. 

Suppose further $\Delta_f>0$. Since $V_1$ is ample, $(V_1^{n-1}\cdot \Delta_f)>0$. Fix a component $B_0$ of $V_1$ intersecting $\Delta_f$ and take the normalization of $B_0$ followed by the inclusion map:
\begin{equation}\label{mm1}
\sigma_1: X_1=\widetilde{B_0}\rightarrow B_0\hookrightarrow X_0=X.
\end{equation}
Then one can get a commutative diagram (by the universal property of normalization):
$$f_1^*(\sigma_1^*H)=\sigma_1^*(f^*H)\sim q\sigma_1^*H,$$
which means that this lifting $f_1$ is polarized by $\sigma_1^* H$ and $\deg f_1=q^{n-1}$.
Pulling back Equation (\ref{mm0}) along the map $\sigma_1$, we have the following:
\begin{equation}\label{mm2}
K_{X_1}+D_1:=\sigma_1^*(K_X+D_0)=\sigma_1^*(f^*(K_X+D_0)+\Delta_f)=f_1^*(K_{X_1}+D_1)+\sigma_1^*\Delta_f.
\end{equation}
Here, $D_1=\textup{Diff}(D_0)$ contains the reduced Weil divisor $\sum_{j\ge 2}\textup{Supp}\,\sigma_1^*V_j$ (cf.~ \cite[Corollary 3.11]{shokurov91}, \cite[Corollary 16.7]{kollar1992flips} and Lemma \ref{lem1}). Besides, $(X_1,D_1)$ is log canonical by Proposition \ref{normallc}.
Note that $V_j|_{B_0}(j>1)$ is  ample on $B_0$ and the normalization is a finite surjective morphism. Hence, each $\sigma_1^* V_j\,(j>1)$ is nonzero and still ample on $X_1$ (cf.~\cite[Theorem 1.37]{kollar2008birational}). Moreover, $\textup{Supp}\,\sigma_1^*V_j\,(j>1)$ is connected since $\dim X_1=n-1\ge 2$  (cf.~\cite[Corollary 7.9]{hartshorne2013algebraic}).

By the choice of $B_0$, $\Delta_f|_{B_0}$ is a nonzero effective $\mathbb{Q}$-divisor on $B_0$. Since the normalization is  finite, $\sigma_1^*\Delta_f$ is a nonzero effective $\mathbb{Q}$-divisor on $X_1$. Repeatedly, we fix an (integral) irreducible component $B_1$ of $\textup{Supp}\,\sigma_1^* V_2~(\le D_1)$ intersecting $\sigma_1^*\Delta_f$. This is possible since $\sigma_1^*V_2$ is ample on $X_1$. With $f$ replaced by its power, we may assume $f_1^{-1}(B_1)=B_1$. Since $f_1$ is polarized, by Corollary \ref{remcariter}, we have $f_1^*B_1=qB_1$. Taking the normalization of $B_1$ followed by the inclusion map, we get $\sigma_2: X_2\rightarrow X_1$ with $\dim X_2=n-2$. Similarly, $f_1|_{B_1}$ lifts to a polarized endomorphism $f_2$ of $X_2$ with $\deg f_2=q^{n-2}$.

In general, let $\sigma_i: X_i\rightarrow X_{i-1}~(1\le i\le n-2)$ be the normalization of an (integral) irreducible component $B_{i-1}$ of $\textup{Supp}\,(\sigma_1\circ\cdots\circ\sigma_{i-1})^* V_i~(\le D_{i-1})$ intersecting $(\sigma_1\circ\cdots\circ\sigma_{i-1})^*\Delta_f$ followed by an inclusion map. Then we get  $f_i: X_i\rightarrow X_i$ polarized by the pullback $(\sigma_1\circ\cdots\circ\sigma_i)^*H$, $\dim X_i=n-i$ and $\deg f_i=q^{n-i}$ (cf.~\cite[Lemma 2.1]{nakayama2010polarized}).
Let $\sigma$ be the composition,
$$\sigma:=\sigma_1\circ\cdots\sigma_{n-2}: X_{n-2}\longrightarrow X_0=X.$$

Now, $S:=X_{n-2}$ is a normal surface with following ample reduced divisors on $S$:
$$C_i:= \textup{Supp}\,\sigma^*V_i~(n-1\le i\le s).$$
After replacing $f$ by its power, we may assume $f_{n-2}^*C_{ij}=q C_{ij}$ for each irreducible component $C_{ij}$ of $C_i~(i\ge n-1)$. On the one hand, using the log ramification divisor formula for the pair $(S,C:=\sum_{i\ge n-1} C_i)$ (cf.~\cite[Theorem 11.5]{iitaka1982algebraic}), we have
\begin{equation}\label{mm3}
K_S+C=f_{n-2}^*(K_S+C)+\Delta.
\end{equation}
Here, $\Delta$ is an effective (integral) Weil divisor, sharing no common components with $C$. On the other hand, pulling back Equation (\ref{mm0}) along the map $\sigma$, we get
\begin{equation}\label{mmnznl}
K_S+D_{n-2}:=\sigma^*(K_X+D_0)=f_{n-2}^*(K_S+D_{n-2})+\sigma^*\Delta_f.
\end{equation}

Comparing Equation (\ref{mm3}) with (\ref{mmnznl}), we get $C\le D_{n-2}$ (cf.~ \cite[Corollary 3.11]{shokurov91} and \cite[Corollary 16.7]{kollar1992flips}), and thus $M:=D_{n-2}-C\ge 0$. By assumption, the number of $C_i$, $r\ge s-(n-2)\ge 3$. 

We will show that $s=n+1$. Indeed, note that each $C_i~(n-1\le i\le s)$ is connected and ample, so $(C_i\cdot C_j)>0$ for $n-1\le i,j\le s$. Let $D=\sum_{j=n-1}^{n+1}C_j\le C$. Then the dual graph of $D$  contains a loop.  By \cite[Lemma 2.8]{zhang2014invariant}, $K_S+D\sim 0$. Back to Equation (\ref{mmnznl}), we have the following:
\begin{equation}\label{mm520}
M+C-D\equiv_w f_{n-2}^*(M+C-D)+\sigma^*\Delta_f.
\end{equation}
Since $\sigma^*\Delta_f$ is an effective $\mathbb{Q}$-divisor, by Lemmas \ref{lemnumer} and \ref{mmnxbxhwa}, $M+C-D\equiv_w 0$ and $\sigma^*\Delta_f=0$.  Thus, $M=C-D=0$ by the effectivity of $M$ and $C-D$. As a result, we have $D_{n-2}=C=D$, which in turn implies that $s=n+1$.

Recall our initial assumption that $\Delta_f\neq 0$. Then by our choice of each $B_i$ and $\sigma$, $\sigma^*\Delta_f$ is a nonzero effective $\mathbb{Q}$-divisor on $S$. However, by Equation (\ref{mm520}) and Lemma \ref{mmnxbxhwa}, we have already got $\sigma^*\Delta_f=0$, a contradiction. Therefore, $\Delta_f=0$ and the ramification divisor of $f$, $R_f=(q-1)D_0$. By the purity of branch loci, $f$ is \'etale outside $(\bigcup V_j)\cup f^{-1}(\textup{Sing}\,X)$. Further, since $\Delta_f=0$, we get $K_X+D_0\equiv 0$ (cf.~ Equation (\ref{mm0}) and Lemma \ref{lemnumer}), which completes the proof of Proposition \ref{promm100}.
 \end{proof}
 
\begin{remark}\label{mm2.24}
\textup{For Proposition \ref{promm100}, when the equality holds, i.e.\,$s=n+1$, we claim that $D_1$ (and then each $D_i$) is reduced. With the same symbols given above, we follow the ideas and steps of \cite[the proof of Theorem 1.1]{zhang2014invariant} to prove it as follows.}
\end{remark}
\begin{proof}
Note that $\Delta_f=0$. Suppose $\Gamma_1=mF_1\le D_1$ is a  non-reduced fractional component with $m<1$.                                                                                                                                                                                                                                                                                                                                                                                                                                                                                                 Since $\sigma_1^*V_2$ is ample on $X_1$, the intersection number $(\Gamma_1\cdot (\sigma_1^*V_2)^{\dim X_1-1})>0$. Therefore, there exists an (integral) irreducible component of $\textup{Supp}\,\sigma_1^*V_2\,(\le D_1)$, intersecting $F_1$. Let this component be our new $B_1$ as in the proof of Proposition \ref{promm100}. Then taking the normalization of $B_1$ followed by the inclusion map, we get
$$K_{X_2}+D_2=f_2^*(K_{X_2}+D_2).$$ 
Then we get the nonzero pullback $\Gamma_2':=\sigma_2^*\Gamma_1\le D_2$  by the choice of $B_1$. Fix an irreducible component $\Gamma_2$ of $\Gamma_2'$, which has no common components with $\textup{Supp}\,\sigma_1^*V_j(j\ge n-1)$.

In general, suppose we have fixed a component $\Gamma_k$ of the pullback of $\Gamma_{k-1}$. Consider the pullback of $V_{k+1}$, which is ample on $X_k$. Choosing the (integral) irreducible component $B_k$ which intersects $\Gamma_k$, we take the normalization of $B_k$ followed by the inclusion.

Finally, we get a component $\Gamma_{n-2}\le D_{n-2}$ on $S$. Since we have proved that $\Delta_f=0$ for the case when $s=n+1$ and it is independent of our choice of $B_i$ and $\sigma$, we get
$$K_S+D_{n-2}=f_{n-2}^*(K_S+D_{n-2}).$$
Note that $K_S+D_{n-2}$ is the pullback of $K_X+D_0$ under our new $\sigma$. Similarly, comparing the above equation with Equation (\ref{mm3}) and (\ref{mmnznl}), we have $M=D_{n-2}-C\ge 0$ and 
$f_{n-2}^*M=M$. Therefore, $M$ is numerically trivial (cf.~ Lemma \ref{lemnumer}), which in turn gives $M=0$. Therefore, $C=D_{n-2}$ and such $\Gamma_{n-2}$ cannot exist since $C$ is the support of the pullback of $\sum_{j=n-1}^{n+1}V_j$ and thus there is no more space for $\Gamma_{n-2}$. As a result, our assumption is absurd.
\end{proof}
The following lemma is known to Iitaka, Sommense, Fujimota and Nakayama (cf.~\cite[Lemma 3.7.1]{nakayama2008complex}), we rewrite it for the convenience of readers.
\begin{lem}[cf.~\cite{nakayama2008complex}]\label{psecod1}
Let $X$ be a normal projective variety of dimension $n$ and $f:X\rightarrow X$ an endomorphism with $\deg (f)\ge 2$. Suppose the canonical divisor $K_X$ is a pseudo-effective $\mathbb{Q}$-Cartier divisor. Then $f$ is \'etale in codimension one.
\end{lem}

\section{Proof of Theorem \ref{thm1.1}}\label{section3}
~~~~In this section, we begin with our proof of Theorem \ref{thm1.1}.
We follow the steps in \cite[The proof of Theorem 1.3]{zhang2014invariant} for the case of polarized endomorphisms. The main idea is to run MMP and reduce the dimension of $X$ gradually. Then we use the induction. In the beginning, we introduce a key proposition (cf.~\cite[The proof of Lemma 2.8]{zhang2014invariant}).
\begin{pro}[cf.~\cite{zhang2014invariant}]\label{prou1}
Suppose $X$ is a rationally connected variety with at worst klt singularities. Then the Picard group $\textup{Pic}(X)$ is torsion free and $\textup{Pic}(X)\cong \textup{NS}(X)$.
\end{pro}
To prove Proposition \ref{prou1}, we need the following well-known fact.
\begin{lem}\label{lempc}
Suppose $X$ is a normal projective variety. Then the set of $\mathbb{C}$-points $X(\mathbb{C})$ is path-connected with respect to complex topology.
\end{lem}
Now, we prove Proposition \ref{prou1}.
\begin{proof}
Take any resolution $g: S\rightarrow X$. Then since $g$ is birational, $S$ is smooth and rationally connected by Proposition \ref{prorc}. Also, Theorem \ref{scc} tells us that $S$ has a trivial fundamental group. By Lemma \ref{univ}, the irregularity of $S$, $q(S)=0$. Now, since $X$ is klt, we have $q(X)=q(S)=0$ (cf.~\cite[Theorem 5.22 and Definition 5.8]{kollar2008birational}), which implies that $\textup{Pic}^\circ(X)=0$. Moreover, by Lemma \ref{lsco}, $\pi_1(X)\cong\pi_1(S)=\{1\}$.

Now, since $\textup{Pic}^\circ(X)=0$, $\textup{Pic}(X)\cong \textup{NS}(X)$, which is finitely generated. Suppose there exists some invertible sheaf $\mathcal{L}=\mathcal{O}(D)$, where $D$ is Cartier such that $nD\sim 0$ for some (minimal) integer $n\ge 1$. Then there exists an unramified cyclic cover $p: X_{n,D}\rightarrow X$ of degree $n$, which is a finite \'etale morphism (cf.~\cite[Definition 2.49]{kollar2008birational}).
Since $X$ is normal and projective, $X(\mathbb{C})$ is path-connected by Lemma \ref{lempc}. Then $\pi_1(X)=\{1\}$ implies that $X$ is simply connected with respect to complex topology and thus algebraically simply connected, which means there does not exist any nontrivial connected finite \'etale cover. Therefore, $n=1$ and $\mathcal{O}(D)$ is trivial. In conclusion, the Picard group of any rationally connected variety with only klt singularities over $\mathbb{C}$ is torsion free. This completes the proof of Proposition \ref{prou1}.
\end{proof}
Now, we use the induction on $\dim X$ to prove Theorem \ref{thm1.1}.
\subsection{The Case $\dim X=1$}
\begin{lem}\label{lemdimx=1}
Theorem \ref{thm1.1} holds for the case when $\dim X=1$.
\end{lem}
In this case, $X$ being normal is equivalent to $X$ being smooth. We give the proof of Lemma \ref{lemdimx=1} as follows.
\begin{proof}
It's well known that a curve with the genus $g(X)>1$ does not admit any non-isomorphic surjective endomorphisms. Since $f:X\rightarrow X$ is int-amplified, $\deg f>1$ (cf.~\cite[Lemma 3,7]{meng2017building}). Hence, there are two cases.

\textbf{Case (a)}. The genus of $X$, $g(X)=1$, i.e.\,$X$ is an elliptic curve. In this case, Theorem \ref{thm1.1} (1) naturally holds, since $c=0$ by Hurwitz's Theorem. For Theorem \ref{thm1.1} (2), since $\textup{NS}(X)\cong\mathbb{Z}$ for the elliptic curve $X$ (cf.~\cite[pp.165 Corollary 2]{mumford2008abelian}), $f^*|_{\textup{NS}_{\mathbb{Q}}(X)}$ is diagonalizable over $\mathbb{Q}$. Further, Theorem \ref{thm1.1} (3) and (4) cannot happen in this case.
\setlength\parskip{8pt}

\textbf{Case (b)}. The genus of $X$, $g(X)=0$. Then $X\cong \mathbb{P}^1$. By Hurwitz's Theorem, 
$$K_X+\sum V_i=f^*(K_X+\sum V_i)+R_{f}',$$
with $R_f'$ effective.  Taking the degree of both sides, we have
$$-2+c=(\deg f)(-2+c)+\deg R_f'.$$
Therefore, $c\le 2$. In this case, Theorem \ref{thm1.1} (2),(3),(4) obviously hold since $\textup{Pic}(\mathbb{P}^1)\cong\mathbb{Z}$. For Theorem \ref{thm1.1} (1),  computing the discrepancy (cf.~\cite[Corollary 2.31]{kollar2008birational}), we see that $(X,\sum V_i)$ is log canonical. Moreover, $X\cong \mathbb{P}^1$ is rational, and then it is uniruled. Thus, we have completed the proof of Theorem \ref{thm1.1} for the case when $\dim X=1$.
\end{proof}
\setlength\parskip{8pt}
\subsection{The Case $\dim X\ge 2$}
Suppose Theorem \ref{thm1.1} holds for those $X'$ with $\dim X'\le n-1$. From now on,  we consider the case  when $\dim X=n\ge 2$. We assume that $V_j\,(1\le j\le c)$ are all the prime divisors in $\textup{TI}_f(X)$, which are contained in the ramification divisors of $f$ (cf.~Lemma \ref{prop3.3}).
According to the hypotheses in our theorem, we may assume $c\ge \rho(X)+n-2\ge 1$ (and hence $K_X$ is not pseudo-effective by Lemma \ref{psecod1}). \setlength\parskip{0pt}

During the proof of Theorem \ref{thm1.1}, for  (1) and (4), we only need one Fano contraction when running MMP while for (2) and (3), we need the end product $Y$ of MMP in Theorem \ref{thm1.10}. The detailed proof is as follows.

\subsubsection{\textbf{Minimal Model Program for $X$}} \label{subsubsectionmmp}

By Theorem \ref{thm1.10}, if $K_X$ is pseudo-effective, then $X$ is $Q$-abelian. From now on, we assume that $K_X$ is not pseudo-effective. Then the minimal model program of $X$ will end with a Fano contraction $X_l\rightarrow Y$ (cf.~[Ibid.]). For each $u_i:X_i\dashrightarrow X_{i+1}$, it is one of the following types: a divisorial contraction, a flip or a Fano contraction of a $K_{X_i}$-negative extremal ray. Furthermore, there exists a positive integer $t$, such that $g_0=f^t$ descends to an int-amplified endomorphism $g_i$ of $X_i$ for each $1\le i\le l+1$.

By the log ramification divisor formula for the pair $(X,\sum V_j)$ (cf.~\cite[Theorem 11.5]{iitaka1982algebraic}),
\begin{equation}\label{eq1}
K_X+\sum V_j=f^*(K_X+\sum V_j)+\Delta_f,
\end{equation}
with $\Delta_f$ effective, having no common components with $\bigcup\text{Supp} V_j$.
Note that each $V_j$ is $f^{-1}$-invariant and thus $g_0^{-1}$-invariant, i.e.\,$\textup{TI}_f(X)\subseteq \textup{TI}_{g_0}(X)$. 

For each $u_i$, the Picard number  will decrease by one if $u_i$ is a divisorial contraction or a Fano contraction while the Picard number will be the same if $u_i$ is a flip (cf.~\cite[Proposition 3.36, 3.37]{kollar2008birational}). This point provides the possibility for our induction.

We  focus on a specific case. When running MMP, once there appears a Fano contraction, one stops it immediately. In other words, we consider the composite morphisms:
\[
\xymatrix{
X=X_0\ar@{-->}[r]&X_1\ar@{-->}[r]&\cdots\ar@{-->}[r]&X_i\ar@{-->}[r]&\cdots\ar@{-->}[r]&X_r\ar@{-->}[r]&X_{r+1}=Y_0,
}
\]
where each $u_i: X_i\dashrightarrow X_{i+1}\,(0\le i\le r-1)$ is either a divisorial contraction or a flip and $u_r: X_r\rightarrow Y_0$ is a Fano contraction. Then each $u_i\,(0\le i\le r-1)$ is a birational map and the dimensions of $X_i~(0\le i\le r)$ are the same.

Let $V(1)_j\subseteq X_1$ be the strict transform of $V_j$ under the birational map  $u_0$ if $V_j$ is not exceptional over $X_1$. Then $V(1)_j$ is still a prime divisor. 
Since $g_0$ descends to an int-amplified endomorphism $g_1$ on $X_1$ by Theorem \ref{thm1.10}, $V(1)_j\in\textup{TI}_{g_1}(X_1)$. Using Lemma \ref{prop3.3}, we get $g_1^*(V(1)_j)=\alpha_j V(1)_j$ for some $\alpha_j>1$. 

In general, for each $0\le i\le r-1$, let $V(i+1)_j$ be the strict transform of $V(i)_j$ under the birational map $u_i$ if $V(i)_j$ is not exceptional over $X_{i+1}$. Then $V(i+1)_j\in\textup{TI}_{g_{i+1}}(X_{i+1})$.

 We use $s_i\,(0\le i\le r)$ to denote the cardinality of  $\textup{TI}_{g_i}(X_i)$. It is obvious that $s_0\ge c$. 
 Also, for each $i$, the number of such $V(i)_j$ cannot exceed $s_i$. By the log ramification divisor formula for the pair $(X_i,\sum_j V(i)_j)$ (cf.~\cite[Theorem 11.5]{iitaka1982algebraic}), we get
\begin{equation}\label{eq2}
K_{X_i}+\sum\limits_{j}V(i)_j=g_i^*(K_{X_i}+\sum\limits_{j}V(i)_j)+\Delta(i).
\end{equation}
Here, $\Delta(i)$ is an integral log ramification  divisor having no common components with $\sum V(i)_j$.  
\begin{lem}\label{invariant}
With the same symbols given above, the equality $s_1=s_0-(\rho(X_0)-\rho(X_1))$ holds (and thus, $s_i=s_0-(\rho(X_0)-\rho(X_i))$ for each $i\le r$). In particular, if $c\ge \rho(X_0)+\dim X_0-k$ for some fixed integer $k$, then $s_r\ge \rho(X_r)+\dim X_r-k$. 
\end{lem}
\begin{proof}
 We use the induction to prove our claim. There are two cases for $u_0$:

Case (1): Suppose $u_0$ is a divisorial contraction. Then, since $u_0$ is $g_0=f^t$ equivariant, the exceptional divisor $E$ of $u_0$ lies in $\textup{TI}_{g_0}(X_0)$. Thus, $s_1= s_0-1$ and also $\rho(X_1)=\rho(X_0)-1$.  In this case, we get $s_1=s_0-(\rho(X_0)-\rho(X_1))$.

Case (2): Suppose $u_0$ is a flip. Then $u_0$ is an isomorphism in codimension one. Therefore, $u_0$ does not contract or create any new divisors. Then $s_0=s_1$, and $\rho(X_1)=\rho(X_0)$. In this case, $s_1=s_0-(\rho(X_0)-\rho(X_1))$ still holds.

In any case, we get the same equality for $X_1$. Repeating the above discussion, we complete the first part of our claim. For the second part, note that $s_0\ge c$ and  $\dim X_i=n$ for each $i\le r$. Hence, the inequality follows from the above equality for each $i\le r$.
\end{proof}

\begin{lem}\label{verykeylemma1}
	With the same symbols given above, suppose $\Delta(r)=0$ for $i=r$ in Equation (\ref{eq2}). Then the following statements hold.
	\begin{enumerate}
	\item[\textup{(1)}] $K_{X_r}+\sum_j V(r)_j\equiv 0$.
	\item[\textup{(2)}] If $\textup{Pic}^\circ(X_r)=0$ and every exceptional divisor of the composite map: $\delta: X_0\dashrightarrow X_r$ is contained in $\textup{TI}_f(X_0)$, where $\delta=u_{r-1}\circ\cdots\circ u_0$, then  $\Delta_f=0$ in Equation (\ref{eq1}) and hence $K_X+\sum V_j\equiv 0$. Moreover, $f$ is \'etale outside $(\bigcup V_j)\cup f^{-1}(\textup{Sing}\,X)$.	\end{enumerate}
\end{lem}
We will apply the above result to prove Theorem \ref{thm1.1} (4). Before proving Lemma \ref{verykeylemma1}, we refer readers to the following remarks on the assumption of this lemma.
\begin{remark}\label{remark3.6}
\textup{A prime divisor $E$ on $X_0$ is $\delta$-exceptional if it is exceptional over some $X_i$. Suppose $E\subseteq X_{i-1}$ is exceptional over $X_i$. Then for the divisorial contraction $u_{i-1}:X_{i-1}\rightarrow X_i$, the image $u_{i-1}(E)$ is a codimension $\ge 2$ closed subset of $X_i$. So it cannot create new divisors when mapped into $X_{i+1}$ no matter $u_i$ is  a divisorial contraction or a flip, since divisorial contractions are birational and flips are isomorphic in codimension one.  Besides, we define the pullback of $\delta$ (cf.~Equation (\ref{eq3})) as the composite of $u_0^*\circ u_1^*\circ\cdots\circ u_{r-1}^*$. Moreover, $\textup{NS}(X_0)$ is generated by the pullback of generators of $\textup{NS}(X_r)$ and all the $\delta$-exceptional divisors (cf.~\cite[Proposition 3.36 and 3.37]{kollar2008birational}).}

\textup{The composite $\delta: X=X_0\dashrightarrow X_r$  is always $g_0$-equivariant by the choice of our $g_0$, i.e.\,$g_0$ descends to an int-amplified endomorphism $g_r$ on $X_r$. Hence, every exceptional divisor $E$ of $\delta$ is contained in $\textup{TI}_{g_0}(X_0)$. However, this is not enough to conclude our result, and the assumption of (2) in Lemma \ref{verykeylemma1}  is necessary. We will see later that the assumption of (2) holds when the cardinality $c$ of $\textup{TI}_f(X_0)$ achieves the upper bound $\dim X_0+\rho(X_0)$. }
\end{remark}

\begin{proof} We begin to prove Lemma \ref{verykeylemma1}.
Indeed, Lemma \ref{verykeylemma1} (1)  follows immediately from Lemma \ref{prop3.3}. Now, we  prove  (2). 
 Since $\textup{Pic}^\circ(X_r)=0$, (1) implies that 
 \begin{equation}\label{eq010}
 	K_{X_r}+\sum V(r)_j\sim_\mathbb{Q} 0.
 \end{equation}
 Write down the well-defined $\delta$-pullback formula as follows:
\begin{equation}\label{eq3}
K_X+\sum V_j=\delta^*(K_{X_r}+\sum V(r)_j)+E_2-E_1.
\end{equation}
Here, $E_k=\sum_l E_{kl}~(k=1,2)$ are effective $\delta$-exceptional divisors and each $E_{kl}$ lies in $\textup{TI}_{f}(X_0)$ by assumption. Hence, we may assume $f^*E_k=\sum_l a_{kl}E_{kl}$ with $a_{kl}>1$. Recall the log ramification divisor formula for $(X,\sum V_j)$ (cf.~Equation (\ref{eq1})):
$$K_X+\sum V_j=f^*(K_X+\sum V_j)+\Delta_f$$
with $\Delta_f$ effective, having no common components with $\bigcup V_i$. In summary, Supp\,$E_1$ and Supp\,$(E_2+\Delta_f)$  have no common components and Supp\,$(E_1+E_2)\subseteq\bigcup V_j$, since $V_j\,(1\le j\le c)$ are all the prime divisors in $\textup{TI}_f(X)$. By Equation (\ref{eq1}), (\ref{eq010}) and (\ref{eq3}),\begin{equation}
E_2-E_1-\Delta_f\sim_{\mathbb{Q}}f^*(E_2-E_1)=\sum a_{2l}E_{2l}-\sum a_{1l}E_{1l}.
\end{equation}
Therefore, we get the following equation with both sides effective:
\begin{equation}\label{eq555}
\Delta_f+\sum(a_{2l}-1)E_{2l}\sim_{\mathbb{Q}}\sum(a_{1l}-1)E_{1l}.
\end{equation}

We claim that $E_1=E_2=\Delta_f=0$. Suppose the contrary that the claim does not hold. Let  $F_k:=\sum_l(a_{kl}-1)E_{kl}\ge 0$ for $k=1,2$. On the one hand, since Supp\,$F_1$ and Supp\,$(\Delta_f+F_2)$ have no common components, Equation (\ref{eq555}) gives us that $\kappa(X,F_1)\ge 1$. On the other hand, we take the graph of the birational map $\delta: \xymatrix{
X\ar@{-->}[r]&X_r}$:
$$
\xymatrix{
&\Gamma\ar[dl]_{p_1}\ar[dr]^{p_2}&\\
X\ar@{--}[r]&\ar@{-->}[r]&X_r
}
$$
Here, the projections $p_1$  and $p_2$ are birational morphisms. As we mentioned above, $F_k$ is $\delta$-exceptional for $k=1,2$; $\delta$ is a composite of flips or birational morphisms and thus $\delta$ does not extract divisors. Therefore, $p_{2*}p_1^*F_1=0$. As a result, we have the following:
\begin{align*}
\textup{H}^0(X,\mathcal{O}_X(mF_1))=\textup{H}^0(\Gamma, p_1^*\mathcal{O}_X(mF_1))=\textup{H}^0(X_r, p_{2*} p_1^*\mathcal{O}_X(mF_1))=\textup{H}^0(X_r,\mathcal{O}_{X_r})=k
\end{align*}
for any $m>0$. The first equality is due to projection formula. By definition, the Iitaka dimension $\kappa(X,F_1)=0$, a contradiction. Therefore, our claim holds.

Since $\Delta_f=0$ in Equation (\ref{eq1}), applying Lemma \ref{prop3.3} to Equation (\ref{eq1}),  $K_X+\sum V_j\equiv 0$. Moreover, $\Delta_f=0$ implies that the ramification divisor of $f$ consists only of  these $V_j$'s. By the purity of branch loci, $f$ is \'etale outside $(\bigcup V_j)\cup f^{-1}(\textup{Sing}\,X)$, which completes the proof of Lemma \ref{verykeylemma1}.
\end{proof}

\subsubsection{\textbf{The Case of $\dim Y_0=0$}}\label{partdim0}
~~~~In this case, $\rho(X_r)=1$ and $-K_{X_r}$ is ample. Then $X_r$ is log $Q$-Fano and thus rationally connected (cf.~\cite{zhang2006rational}). Since $X$ and $X_r$ are birational, $X$ is also rationally connected (and hence uniruled) by Proposition \ref{prorc}. Besides, $X$ is simply connected with respect to complex topology (cf.~Theorem \ref{scc} and Lemma \ref{lsco}) and $X_r$ also has only klt singularities (cf.~\cite[Corollary 3.42, 3.43]{kollar2008birational}). Then, by Proposition \ref{prou1}, $\textup{Pic}^\circ(X_r)=\textup{Pic}^\circ(X_0)=0$, $\textup{Pic}(X_0)\cong \textup{NS}(X_0)$ and $\textup{Pic}(X_r)\cong \textup{NS}(X_r)$. 

We fix an ample divisor $H_{X_r}$ on $X_r$. Since $\rho(X_r)=1$, we get $g_r^*H_{X_r}\equiv a H_{X_r}$ and also $a>1$ (cf.~Lemma \ref{prop3.3}). Hence $g_r$ is polarized by \cite[Lemma 2.3]{meng2018building}. Further, since $\textup{Pic}^\circ(X_r)=0$, with $H_{X_r}$ replaced by its power, we may assume $g_r^*H_{X_r}\sim aH_{X_r}$. By Proposition \ref{prorc} and Remark \ref{remark3.6}, $\textup{Pic}(X_0)\otimes_{\mathbb{Z}}\mathbb{Q}\cong \textup{NS}_{\mathbb{Q}}(X_0)$ as a $\mathbb{Q}$-vector space, is spanned by $\delta$-exceptional divisors and  $\delta^*H_{X_r}$  over $\mathbb{Q}$. 
On the one hand, $g_0^*(\delta^*(H_{X_r}))=\delta^*(g_r^*(H_{X_r}))\sim a\delta^*(H_{X_r})$ for some $a>1$.  On the other hand, for each $\delta$-exceptional prime divisor $E$, it is contained in $\textup{TI}_{g_0}(X_0)$ by the choice of $g_0$ and $g_0^*E=\beta E$ for some $\beta>1$. Therefore, $g_0^*|_{\textup{Pic}X}$ is diagnolizable over $\mathbb{Q}$. 
Since $g_0$ is a positive power of $f$, this completes the proof of Theorem \ref{thm1.1} (2) and (3) for the case when $\dim Y_0=0$.
\setlength\parskip{8pt}

Suppose $c\ge \rho(X)+n$ from now on. Then $s_r\ge n+\rho(X_r)=n+1$ (cf.~Lemma  \ref{invariant}). \setlength\parskip{0pt}
Since $g_r$ is polarized and each $V(r)_j$ is ample, by Proposition \ref{promm100}, we get
\begin{equation}\label{111}
s_r=n+1,~~~K_{X_r}+\sum\limits_{j=1}^{n+1}V(r)_j\equiv 0.
\end{equation}
Therefore,  by \cite[CH.~\Rmnum{13}, Theorem 4.6]{berthelot2006theorie} (and $\textup{Pic}^\circ(X_r)=0$), we also get Equation (\ref{eq010}).
Since  $s_r=n+1=n+\rho(X_r)$, by Lemma \ref{invariant},  $s_i=n+\rho(X_i)$ for each $0\le i\le r$. The equality for the case when $i=0$ in turn forces $s_0=c$. Thus, $\textup{TI}_f(X_0)=\textup{TI}_{g_0}(X_0)$, i.e.\,all of these $g_0^{-1}$-invariant prime divisors consist only of $V_j$'s.  This proves Theorem \ref{thm1.1} (1) for the case when $\dim Y_0=0$. 

\setlength\parskip{8pt}

Now we prove Theorem \ref{thm1.1} (4) for the case when $\dim Y_0=0$. Since $s_i=n+\rho(X_i)$ for each $i$, the exceptional divisor of $u_i$ is contained in $\sum_jV(i)_j$ if $u_i$ is a divisorial contraction. Also, since $X_r$ and also $X$ are rationally connected with at worst klt singularities, by Proposition \ref{prou1}, $\textup{Pic}^\circ(X_r)=\textup{Pic}^\circ(X)=0$.\setlength\parskip{0pt} Hence, in this case, Theorem \ref{thm1.1} (4) follows from Lemma \ref{verykeylemma1}.
\setlength\parskip{8pt}

\subsubsection{\textbf{The Proof of Theorem \ref{thm1.1} (1) for the Case $\dim Y_0\ge 1$}}\label{part2}\setlength\parskip{0pt}
We still assume that $\dim Y_0\ge 1$ and $c\ge \rho(X)+\dim X$. 
Actually, the second part of (1) follows immediately. Indeed, if $c\ge 1$, then the ramification divisor $R_f>0$, which means $f$ is not \'etale in codimension one. By Lemma \ref{psecod1}, $K_X$ is not pseudo-effective and hence $X$ is uniruled by Remark \ref{remncu}. Moreover, by Lemma \ref{lem1}, the pair $(X,\sum V_j)$ is log canonical.

Recall that Lemma \ref{invariant} gives us the same inequality for each $X_i~(i\le r)$: $s_i\ge n+\rho(X_i)$ if we assume $s_0\ge n+\rho(X_0)$. We want to ask whether the inequality still holds for $Y_0$. If the inequality holds, then we can use the induction on $\dim Y_0$, which proves Theorem \ref{thm1.1} (1). Parallel to Lemma \ref{invariant}, we introduce the following lemma, from which, the first part of Theorem \ref{thm1.1} follows immediately. 

To emphasize the notation for the Fano contraction $u_r:X_r\rightarrow Y_0$,  let $s_{Y_0}$ be cardinality of $\textup{TI}_{g_{Y_0}}(Y_0)$ and $g_{Y_0}:=g_{r+1}$ the int-amplified endomorphism of $Y_0$, which $g_0$ descends to. 
\begin{lem}\label{invariant1}
With the same symbols given above, for the Fano contraction $u_r:X_r\rightarrow Y_0$, the inequality $s_r\le s_{Y_0}+\dim X_r-\dim Y_0+1$ holds. Further, if $s_r\ge n+\rho(X_r)-k$ for some fixed integer $k$, then $s_{Y_0}\ge \dim Y_0+\rho(Y_0)-k$. In particular, if $k=0$, then $s_{Y_0}=\dim Y_0+\rho(Y_0)$ and also $s_r=n+\rho(X_r)$. 	
\end{lem}

To prove Lemma \ref{invariant1}, we need to do some preparations.
By Lemma \ref{densepoint}, the periodic points of $g_{Y_0}$ are dense. Let $y_0\in Y_0$ be a general $g_{Y_0}$-periodic point. Replacing $g_{Y_0}$ (and $g_i$) by its power, we may assume $g_{Y_0}(y_0)=y_0$. Let $W_0:=u_r^{-1}(y_0)\subseteq X_r$. Then $W_0$, as a general fibre of $u_r$, is a Fano variety of dimension $n-\dim Y_0\ge 1$, since the canonical divisor of a general fibre of a Fano contraction is anti-ample.  Restricting $g_r$ to $W_0$, we get a surjective endomorphism of $W_0$ commuting with $u_r$ and $g_{Y_0}$:
$$g_{W_0}:=g_r|_{W_0}:W_0\rightarrow W_0.$$

\begin{claim}\label{claim000}
With the same symbols given above, $g_{W_0}$ is polarized.
\end{claim}
\begin{proof}
Since $u_r: X_r\rightarrow Y_0$ is a Fano contraction, the relative Picard number $\rho(X_r/Y_0)=1$. Therefore, pulling back a fixed ample divisor $H_r$ on $X_r$, we get 
\begin{equation}\label{mmwxn}
g_r^* H_r\equiv  kH_r \mod(u_r^*\,\textup{N}^1(Y_0)).
\end{equation}
Since $g_r$ is int-amplified, each eigenvalue of $g_r^*|_{\textup{N}^1(X_r)}$ is of modulus greater than $1$ (cf.~Lemma \ref{prop3.3}). Besides, the operation $g_r^*|_{\textup{N}^1(X_r)}$ induces an operation on the space $\textup{N}^1(X_r)/u_r^*\,\textup{N}^1(Y_0)$, with all the eigenvalues of modulus greater than $1$. Thus, $k>1$. 

Now, since $u_r^*D|_{W_0}=0$ for any Cartier divisor $D$ on $Y_0$, restricting Equation (\ref{mmwxn}) to a general fibre $W_0$, we get 
$g_{W_0}^* (H_r|_{W_0})\equiv k H_r|_{W_0},$
with $H_r|_{W_0}$ ample on $W_0$. This together with $k>1$ proves that $g_{W_0}$ is polarized (cf.~\cite[Lemma 2.3]{nakayama2010polarized}).
\end{proof}\setlength\parskip{0pt}

Now, we begin to prove Lemma \ref{invariant1}.
\begin{proof}
	We divide these $V(r)_i(1\le i\le s_r)$ into two groups.
	
\textbf{Case (1)}: Suppose $u_r|_{V(r)_j}:V(r)_j\rightarrow Y_0~(j\le s_r(1))$ is not surjective. Since $u_r$ is projective and thus closed, the image $u_r(V(r)_j)$ is a codimension $\ge 1$ closed subset. Take a general point $y_0\in Y_0$ such that $y_0$ does not lie in the image of any $V(r)_j$ for $j\le s_r(1)$. Then one may easily get $V(r)_j|_{W_0}=0$. Since $\textup{Pic}(X_r/Y_0)=1$, any two contracted curves are proportional (under the numerical equivalence). Hence, for any curve $C$ lying in a fibre of $X_r\rightarrow Y_0$,  $(V(r)_j\cdot C)=0$. Since $V(r)_j$ is $\mathbb{Q}$-Cartier, we may take a suitable $k_j'\in\mathbb{N}$ such that $k_j' V(r)_j$ is Cartier.

By Cone Theorem (cf.~\cite[Lemma 3-2-5]{kawamata1987introduction} or \cite[Theorem 3.7 (4)]{kollar2008birational}), $k_j'V(r)_j=u_r^*G_j'$ for some effective Cartier divisor $G_j'$ on $Y_0$.
\setlength\parskip{0pt}Then the image $G_j:=u_r(V(r)_j)$, which is the support of $G_j'$, is a prime divisor on $Y_0$. Since $Y_0$ is normal, on the smooth locus of $G_j$, the pullback of $G_j$ is an integral divisor. Taking the closure, we get $u_r^*G_j=k_j V(r)_j$ for some $k_j\in\mathbb{N}$. 
Moreover, by the commutative diagram, $g_{Y_0}^{-1}(G_j)=G_j$ set-theoretically for each $j\le s_r(1)$. Applying the inductive hypothesis on $\dim Y_0$, the following holds:
\begin{equation}\label{mmmmwan1}
s_r(1)\le s_{Y_0}\le\dim Y_0+\rho(Y_0).
\end{equation}
Here,  the case when $s_r(1)=0$ is allowed during our discussion.
\setlength\parskip{8pt}

\textbf{Case (2)}: Suppose $u_r|_{V(r)_j}: V(r)_j\rightarrow Y_0~(s_r(1)<j\le s_r)$ is surjective. Note that the general fibre $W_0$ is not contained in any $V(r)_j$ for $s_r(1)<j\le s_r$. Let $s_r(2):=s_r-s_r(1)$ and fix an ample divisor $H_{X_r}$ on $X_r$, which does not lie in $u_r^*\,\textup{N}^1(Y_0)$. Since $\rho(X_r/Y_0)=1$, we have the following numerical property:\setlength\parskip{0pt}
$$V(r)_j\equiv t_j H_{X_r} \mod (u_r^*\,\textup{N}^1(Y_0))$$
for some $t_j>0$, and thus $V(r)_j|_{W_0}$ is ample on $W_0$.  Moreover, $V(r)_j|_{W_0}$ is $g_{W_0}^{-1}$-invariant on $W_0$ for each $s_r(1)<j\le s_r$. Note that we also allow the case $s_r(2)=0$.

By Claim \ref{claim000}, $g_{W_0}$ is polarized. If $\dim W_0\ge 2$, then by Proposition \ref{promm100}, we get  $s_r(2)\le \dim W_0+1$; if $\dim W_0=1$, then by Lemma \ref{lemdimx=1}, we get $s_r(2)\le 2=\dim W_0+1$. In any case, we get the following inequality:
\begin{equation}\label{eq17zzz}
	s_r(2)\le \dim W_0+1.
\end{equation}
\setlength\parskip{0pt}
Combining the inequality (\ref{eq17zzz}) with
 the first inequality of  (\ref{mmmmwan1}), the following holds:\setlength\parskip{0pt}
 \begin{align}\label{eq17777}
s_r=s_r(1)+s_r(2)\le s_{Y_0}+s_r(2)\le s_{Y_0}+\dim W_0+1=s_{Y_0}+\dim X_r-\dim Y_0+1.
 \end{align}
This completes the first part of our lemma. 

If $s_r\ge \dim X_r+\rho(X_r)-k$ for some integer $k$, then by Equation (\ref{eq17777}) and the fact that $\rho(Y_0)=\rho(X_r)-1$, we have  $s_{Y_0}\ge \dim Y_0+\rho(Y_0)-k$.
 In particular, if $k=0$, then combining   (\ref{eq17zzz}) with the second inequality of (\ref{mmmmwan1}), we see that,
 $$n+\rho(X_r)\le s_r=s_r(1)+s_r(2)\le \dim Y_0+\rho(Y_0)+\dim W_0+1=n+\rho(X_r).$$
Therefore, all the inequalities are equalities, and thus for each $i$, $s_i=n+\rho(X_i)$. This proves Lemma \ref{invariant1} and also Theorem \ref{thm1.1} (1) for the case when $\dim Y_0\ge 1$ since $c\le s_0$.
\end{proof}

At the end of this part, the inductive hypothesis on $Y_0$ implies the following:
\begin{equation}\label{eq111111}
	s_r(1)=\dim Y_0+\rho(Y_0),~~~K_{Y_0}+\sum\limits_{i=1}^{s_r(1)}G_i\sim_{\mathbb{Q}}0.
\end{equation}


\subsubsection{\textbf{The Proof of Theorem \ref{thm1.1} (2) and (3) for the Case When $\dim Y_0\ge 1$}}
In this part, we shall use Theorem \ref{thm1.10} to prove  Theorem  \ref{thm1.1} (2), (3). Recall that we have such a $g_0$-equivariant relative MMP over $Y$
\[
\xymatrix{
X=X_0\ar@{-->}[r]&X_1\ar@{-->}[r]&\cdots\ar@{-->}[r]&X_i\ar@{-->}[r]&\cdots\ar@{-->}[r]&X_r\ar@{-->}[r]&X_{r+1}=Y_0\ar@{-->}[r]&Y.}
\]
Here, we continue running MMP from $Y_0$ ($X_r\rightarrow Y_0$ is the first Fano contraction) and terminate with the end product $Y$. By the above discussion, we first consider the case when $Y=Y_0$, i.e.\,assume the MMP has only one Fano contraction. Suppose
$$s_0\ge c\ge \dim X+\rho(X)-2.$$
Then by Lemmas \ref{invariant} and \ref{invariant1}, $s_{Y_0}\ge\dim Y_0+\rho(Y_0)-2$. 
Similarly, if $c\ge\dim X+\rho(X)-1$, then  $s_{Y_0}\ge\dim Y_0+\rho(Y_0)-1$.
Now we may start from $Y_0$ and continuously run MMP as mentioned in Theorem \ref{thm1.10}.
\begin{claim}\label{ellip}
For the case when $s_0\ge\dim X+\rho(X)-2$, there are only two choices for the end product $Y$ of MMP: $Y$ is either an elliptic curve or a point.
\end{claim}
\begin{proof}
Suppose the contrary that the claim does not hold. If $\dim Y\ge 2$, then $s_Y\ge \dim Y+\rho(Y)-2\ge 1$, which means the int-amplified endomorphism $g_Y$ has ramification divisors. Since $g_Y$ is not \'etale in codimension one, $K_Y$ is not pseudo-effective by Lemma \ref{psecod1}. However, our end product $Y$ is $Q$-abelian and then $K_Y\sim_{\mathbb{Q}} 0$ (cf.~Theorem \ref{thm1.10}), a contradiction. Therefore, $\dim Y=1$ and by Lemma \ref{lemdimx=1}, $Y$ is either elliptic or rational. If $Y\cong\mathbb{P}^1$, then $K_Y$ is  not pseudo-effective.  By Theorem \ref{thm1.10} again, we can continue to contract $Y$ into a single point. This proves Claim \ref{ellip}.
\end{proof}
\setlength\parskip{8pt}

\textbf{The proof of Theorem \ref{thm1.1} (3).} If $Y$ is elliptic, then $s_Y=0$ (cf.~Lemma \ref{lemdimx=1}), contradicting the deduced result $s_Y\ge \dim Y+\rho(Y)-1$. Therefore, $Y$ is a point (cf.~Claim \ref{ellip}).  By Theorem \ref{thm1.10} (2),  $X$ is rationally connected since the whole $X$ is  a fibre. Then, by Proposition \ref{prou1}, $\textup{Pic}^\circ(X)=0$ and then we may identify the Picard group with the N\'eron--Severi group. Further, $X$ is simply connected with respect to complex topology and $\textup{Pic}(X)$ is torsion free.

\setlength\parskip{0pt}
 Theorem \ref{thm1.10} asserts that $g_0^*|_{\textup{Pic}(X)}$ is diagonalizable over $\mathbb{C}$ if and only if so is $g_Y^*|_{\textup{Pic}(Y)}$. Since $Y$ is a point, $g_0^*|_{\textup{Pic}(X)}$ is diagonalizable over $\mathbb{C}$. Furthermore, each eigenvalue $\lambda$ of $g_0^*|_{\textup{Pic}(X)}$ is either an eigenvalue of $g_Y^*|_{\textup{Pic}(Y)}$ or an eigenvalue of $g_i^*|_{\textup{Pic}(X_i)/u_i^*\textup{Pic}(X_{i+1})}$ for some $i$. Since $Y$ is a point, $\lambda$ is an eigenvalue of $g_i^*|_{\textup{Pic}(X_i)/u_i^*\textup{Pic}(X_{i+1})}$ for some $i$. Since $\rho(X_i/X_{i+1})=1$, $\lambda\in\mathbb{Q}$. Indeed, all these eigenvalues are positive integers (cf.~\cite[Lemmas 5.1 and 5.2]{meng2018semi}). Thus, $g_0^*|_{\textup{Pic}(X)}$ is diagonalizable over $\mathbb{Q}$. This completes the proof of Theorem \ref{thm1.1} (3).
\setlength\parskip{8pt}

\textbf{The proof of Theorem \ref{thm1.1} (2).} By Claim \ref{ellip}, there are two choices for the end product $Y$ of MMP. If $Y$ is a point, then similar to Theorem \ref{thm1.1} (3), $X$ is rationally connected and simply connected with respect to complex topology; if $Y$ is an elliptic curve, then this $Y$ is our $E$ in Theorem \ref{thm1.1} (2). By  Theorem \ref{thm1.10}, $\phi: X\rightarrow Y=E$ is holomorphic and equi-dimensional with every fibre irreducible. Then $\phi: X\rightarrow E$ is proper and surjective, the general fibre of which is connected, and thus $\phi$ is a fibration. Further, $g_0=f^t$  descends to $g_E: E\rightarrow E$ of degree $>1$ (cf.~Theorem \ref{thm1.10}). Furthermore, $g_0^*|_{\textup{NS}_{\mathbb{Q}}(X)}$ is diagonalizable over $\mathbb{C}$ if and only if $g_E^*|_{\textup{NS}_{\mathbb{Q}}(E)}$ is diagonalizable over $\mathbb{C}$ (cf.~\cite[Lemma 9.2]{meng2018building}). Since  we have proved $g_E^*|_{\textup{NS}_{\mathbb{Q}}(E)}$ is diagonalizable over $\mathbb{Q}$ by Lemma \ref{lemdimx=1}, $g_0^*|_{\textup{NS}_{\mathbb{Q}}(X)}$ is diagonalizable over $\mathbb{C}$. As in the case of Theorem \ref{thm1.1} (3), all the eigenvalues of $g_0^*|_{\textup{NS}_{\mathbb{Q}}(X)}$ are rational numbers. Thus, $g_0^*|_{\textup{NS}_{\mathbb{Q}}(X)}$ is diagonalizable over $\mathbb{Q}$.\setlength\parskip{0pt}

Finally, we prove that every fibre of $\phi: X\rightarrow E$ is normal (cf.~\cite[The proof of Theorem 1.3]{zhang2014invariant}).  Let $\Sigma=\{e\in E\mid \phi^*e=X_e\textup{~is not normal}\}$, a finite subset of $E$. By \cite[The proof of Lemma 4.7]{nakayama2010polarized}, $g_E^{-1}(\Sigma)\subseteq \Sigma$,
which  implies  $g_E^{-1}(\Sigma)=\Sigma$ since $\Sigma$ is a finite set.  However, $g_E$ is \'etale and it could not have any $g_E^{-1}$-invariant divisors (cf.~Lemma \ref{lemdimx=1}). As a result, $\Sigma=\emptyset$ and Theorem \ref{thm1.1} (2) holds.
\setlength\parskip{8pt}
\subsubsection{\textbf{The Proof of Theorem \ref{thm1.1} (4) for the case $\dim Y_0\ge 1$}}
~~~~We shall prove $K_X+\sum_{j=1}^{n+\rho}V_j\sim_{\mathbb{Q}}0$ and $f$ is \'etale outside $f^{-1}(\textup{Sing}\,X)\cup(\bigcup V_j)$ for the case when $c=\dim X+\rho(X)$. In this part, the first Fano contraction is enough for our proof. As we proved in Subsubsection \ref{part2}, \setlength\parskip{0pt}if $s_0\ge c\ge\dim X+\rho(X)$, then $s_0=c=\dim X+\rho(X)$ and $s_r(2)=\dim W_0+1$, where $W_0$ is a general fibre of $X_r\rightarrow Y_0$.

\begin{claim}
Suppose $c=\dim X+\rho(X)$. Then $K_{W_0}+\sum_{j=1}^{s_r} V(r)_j|_{W_0}\equiv 0$.
\end{claim}
\begin{proof}
As we proved in Lemma \ref{invariant1},  $V(r)_j|_{W_0}$ is ample on $W_0$  for each $j>s_r(1)$. When $c=\dim X+\rho(X)$,  $s_r(2)=\dim W_0+1$.
Besides, by Claim \ref{claim000}, $g_{W_0}=g_r|_{W_0}$ is polarized by the restriction of an ample divisor on $X_r$.  If $\dim W_0=1$, then by Case (b) of Lemma \ref{lemdimx=1}, $K_{W_0}+\sum V(r)_j|_{W_0}\equiv 0$; if $\dim W_0\ge 2$, then applying Proposition \ref{promm100} to the pair $(W_0,g_{W_0})$ with $g_{W_0}^{-1}$-invariant ample divisors $V_j|_{W_0}(j>s_r(1))$, we have the following:
\begin{equation}\label{eq4}
(K_{X_r}+\sum\limits_{j=1}^{s_r}V(r)_j)|_{W_0}\equiv K_{W_0}+\sum\limits_{j=s_r(1)+1}^{s_r(1)+s_r(2)}V(r)_l|_{W_0}\equiv  0.
\end{equation}
Here, $V(r)_j|_{W_0}=0$ when $j\le s_r(1)$ (cf.~Case (1) in the proof of Lemma \ref{invariant1}). We have completed the proof of our claim.
\end{proof}
Recall Equation (\ref{eq2}) for the case when the index $i=r$ as follows:
\begin{equation}\label{mmwxnan}
K_{X_r}+\sum\limits_{j=1}^{s_r} V(r)_j=g_r^*(K_{X_r}+\sum\limits_{j=1}^{s_r}V(r)_j)+\Delta(r).
\end{equation}
Restricting Equation (\ref{mmwxnan}) to the general fibre $W_0$ and then comparing it with Equation (\ref{eq4}), we get $\Delta(r)|_{W_0}\equiv 0$. Since $\Delta(r)$ is effective in the log ramification divisor formula, by Cone Theorem, $\Delta(r)$ is the pullback of some effective $\mathbb{Q}$-divisor $\Delta_{Y_0}\subseteq Y_0$. 
\begin{claim}\label{keyclaim}
$\Delta_{Y_0}=0$, i.e.\,$\Delta(r)=0$.
\end{claim}
Suppose Claim \ref{keyclaim} holds for the time being. Then by Equation (\ref{mmwxnan}), we have
\begin{equation}\label{mmmmwan}
K_{X_r}+\sum_{j=1}^{s_r} V(r)_j=g_r^*(K_{X_r}+\sum_{j=1}^{s_r} V(r)_j).
\end{equation}
Suppose $K_{X_r}+\sum_j V(r)_j\not\equiv 0$. Then from Equation (\ref{mmmmwan}), $1$ is an eigenvalue for the operator $g_r^*|_{\textup{NS}_{\mathbb{Q}}(X_r)}$, a contradiction  (cf.~Lemma \ref{prop3.3}). This forces $K_{X_r}+\sum V(r)_j\equiv 0$ and thus $K_{X_r}+\sum V(r)_j\sim_\mathbb{Q} 0$, since $\textup{Pic}^\circ(X_r)=\textup{Pic}^\circ(X_0)=0$ (cf.~Proposition \ref{prou1} and the proof of Theorem \ref{thm1.1} (3)).
Then Theorem \ref{thm1.1} (4) follows from Lemma \ref{verykeylemma1}.

Now, the only thing we need to do is to prove Claim \ref{keyclaim}.
\setlength\parskip{8pt}

\textit{Proof of Claim \ref{keyclaim}}. We shall construct a generically finite and surjective morphism to $Y_0$. Let $W(s_r+1):=X_r$.  We take four steps to prove the claim. \setlength\parskip{0pt}

\textbf{Step 1.} First, let $\sigma(s_r): W(s_r)\rightarrow X_r$ be the normalization $W(s_r)\rightarrow V(r)_{s_r}$ followed by an inclusion map $i:V(r)_{s_r}\hookrightarrow X_r$. Then $\dim W(s_r)=\dim X_r-1$. It is easy to get the following claim by considering the commutative diagram and applying Lemma \ref{lem3.6}.
\begin{claim}\label{claim6}
With the same symbols given above, there exists an int-amplified endomorphism $g(s_r)$ of $W(s_r)$ such that $\sigma(s_r)\circ g(s_r)=g_r\circ \sigma(s_r)$.
\end{claim}
We return back to Step 1 of our proof for Claim \ref{keyclaim}. Pulling back Equation (\ref{mmwxnan}) along the morphism $\sigma(s_r)$,  we get the following (cf.~Claim \ref{claim6}):
\begin{equation*}
K_{W(s_r)}+D(s_r):=\sigma(s_r)^*(K_{X_r}+\sum V(r)_j)=g(s_r)^*(K_{W(s_r)}+D(s_r))+\Delta(r)_{s_r}.
\end{equation*}
Here, $\Delta(r)_{s_r}$ is the pullback of $\Delta(r)$. Note that each $V(r)_j(j\le s_r(1))$ intersects $V(r)_{s_r}$ (cf.~the proof of Lemma \ref{invariant1}). Indeed, since $k_jV(r)_j=u_r^*G_j$ for each $j\le s_r(1)$, $V(r)_j$ contains all the fibres over $G_j$. Also, $V(r)_{s_r}$ dominates $Y_0$ and hence $V(r)_j$ must intersect $V(r)_{s_r}$. Then the pullback of each $V(r)_j\,(j\le s_r(1))$ survives in $W(s_r)$. The next  claim follows from \cite[Corollary 3.11]{shokurov91} and \cite[Corollary 16.7]{kollar1992flips} (also cf.~\cite[Proposition 2.5]{kollar2013singularity}).

\begin{claim}\label{cliam3.10}
The pair $(W(s_r),D(s_r))$ is log canonical and the support of the pullback  $\sum_{j=1}^{s_r(1)}\sigma(s_r)^*V(r)_{j}$ (as a reduced divisor) is contained in $D(s_r)$, i.e.\,
$$D(s_r)\ge\sum_{j=1}^{s_r(1)}\textup{Supp}\,\sigma(s_r)^*V(r)_j.$$ 
\end{claim}
\begin{proof}
	The first part follows from Proposition \ref{normallc}. For the second part, on the one hand, since $(W(s_r),D(s_r))$ is log canonical, the coefficient of each component in $D(s_r)$ cannot exceed $1$. 
	On the other hand, for each irreducible component $P(j)_i$ of the support of the intersection $V(r)_{s_r}\cap V(r)_j$, note that $(X_r,V(r)_{s_r})$ should be log terminal around $P(j)_i$ (cf.~\cite[Lemma 2.27]{kollar2008birational}). Hence, the coefficient of $P(j)_i$ in $D(s_r)$ is given by (cf.~\cite[Corollaries  3.10 and 3.11]{shokurov91})
	$$1-\frac{1}{m}+\sum\frac{r_j}{m},$$ 	
	where $r_j/m$ is the multiplicity of $\sigma(s_r)^*V(r)_j$ around $P(j)_i$ and $r_j$ is a natural number. 
	Since $V(r)_j\,(j\le s_r(1))$ intersects $V(r)_{s_r}$ as we discussed above, $r_j\ge 1$ and hence the coefficient of $P(j)_i$ in $D(s_r)$ is no less than $1$.
	In conclusion, the total coefficient of each component of $\textup{Supp}\,\sigma(s_r)^*V(r)_j\,(j\le s_r(1))$ in $D(s_r)$  is $1$ and our claim holds.
\end{proof}

We come back to the proof of Claim \ref{keyclaim}. Suppose $\dim X_r-\dim Y_0=1$. Then we go directly to Step 4 with $k_0=s_r$.

\textbf{Step 2.} Suppose $\dim X_r-\dim Y_0\ge 2$. Then $\dim W_0\ge 2$. Therefore, $V(r)_{s_r}|_{W_0}$ is connected since it is ample on $W_0$ (cf.~\cite[Corollary 7.9]{hartshorne2013algebraic}). Further, for each $s_r(1)<j<s_r$, $V(r)_j|_{W_0}$ is also ample and thus nonzero on $W_0$. 

Since $V(r)_{s_r}$ also intersects $V(r)_j$ for each $s_r(1)<j<s_r$ by the ampleness of $V(r)_j$. With  the same proof of Claim \ref{cliam3.10}, the support of $\sigma(s_r)^*V(r)_j(s_r(1)<j<s_r)$, which is a reduced divisor, is contained in $D(s_r)$.

We claim that each $\textup{Supp}\,\sigma(s_r)^*V(r)_j\, (s_r(1)<j<s_r)$ still dominates $Y_0$. Indeed, a general fibre of $V(r)_{s_r}\hookrightarrow X_r\rightarrow Y_0$  
is as the form of $W_0\cap V(r)_{s_r}$. Since each $V(r)_j|_{W_0}\,(s_r(1)<j<s_r)$ is ample on $W_0$,  the restriction $V(r)_j|_{W_0\cap V(r)_{s_r}}$ is still ample on $W_0\cap V(r)_{s_r}$ for $s_r(1)<j<s_r$. Note that this can only be obtained under the condition that $\dim W_0\ge 2$. Therefore, the nonzero divisor $V(r)_j|_{V(r)_{s_r}}$ on $V(r)_{s_r}$ intersects  the general fibre of $V(r)_{s_r}\rightarrow Y_0$ and thus its support dominates $Y_0$. Moreover, normalization is  finite, and hence $\sigma(s_r)^*V(r)_j$ is nonzero and ample when restricted to any general fibre of $W(s_r)\rightarrow Y_0$ for any $s_r(1)<j<s_r$. As in the proof of  Proposition \ref{promm100}, we fix an (integral) irreducible component $\textup{Supp}\,B_{s_r-1}~(\le D(s_r))$ of $\textup{Supp}\,\sigma(s_r)^*V(r)_{s_r-1}$ dominating $Y_0$ and then take the normalization $W(s_r-1)$ of $B_{s_r-1}$ followed by the inclusion to get the next morphism $\sigma(s_r-1):W(s_r-1)\rightarrow W(s_r)$.

\textbf{Step 3.} In general, for each $s_r\ge k\ge s_r-\dim W_0+1=s_r(1)+2$, let $\sigma(k): W(k)\rightarrow W(k+1)$ be the normalization of a fixed (integral) irreducible component $B_k$ of $\textup{Supp}\,V(r)_k|_{W(k+1)}$ dominating $Y_0$, followed by inclusion. Such $B_k$ exists by the induction and the following condition. Since the dimension of the fibre of $W(k+2)\rightarrow Y_0$ is $\dim W_0-s_r+k+1$ which is at least $2$ by  assumption, $(\sigma(s_r)\circ\cdots\circ\sigma(k+2))^*V(r)_k$ intersects $B_{k+1}$ and hence the support of the pullback  $V(r)_k|_{W(k+1)}$  dominates $Y_0$ with the same proof as in Step 2. 

Moreover, for each $k\ge s_r(1)+2$, $B_k$ intersects  $V(r)_j|_{W(k+1)}\,(j\le s_r(1))$ and hence the support of the pullback of $V(r)_j(j\le s_r(1))$ to $W(k)$ (as a reduced divisor) is contained in $D(k)$ (cf.~Claim \ref{cliam3.10}).
Now, $\dim W(k)=\dim X_r-(s_r-k)-1$ and similar to Claim \ref{claim6}, we get an int-amplified endomorphism $g(k)$ of $W(k)$, commuting with each  $g(l)$ for $l>k$. Therefore, we get the following equations:
\begin{equation}\label{eq5}
K_{W(k)}+D(k):=\sigma(k)^*(K_{W(k+1)}+D(k+1))=g(k)^*(K_{W(k)}+D(k))+\Delta(r)_k
\end{equation}
with $k=s_r,s_{r-1},\cdots, k_0:=s_r(1)+2$. For each  $k$, $\Delta(r)_k$ is the pullback of $\Delta(r)$; the pair $(W(k), D(k))$ is log canonical with $D(k)\ge\sum_{j=1}^{s_r(1)}\textup{Supp}\,V(r)_j|_{W(k)}$.

\textbf{Step 4.} Let $R_j\,(j\le s_r(1))$ denote the support of the pullback of $V(r)_j$ to $W(k_0)$, which are all reduced integral divisors (may not be irreducible). Then $D(k_0)\ge \sum_{j =1}^{s_r(1)} R_j$. Let
$$\tau_{k_0}:=u_r\circ\sigma(s_r)\cdots\circ\sigma(k_0): W(k_0)\rightarrow\cdots\rightarrow X_r\rightarrow Y_0.$$
By construction, $\tau_{k_0}$ is dominant and projective. Further, $\dim W(k_0)=\dim Y_0$, and thus $\tau_{k_0}$ is generically finite. By the commutative diagram, we have $\tau_{k_0}\circ g(k_0)=g_{Y_0}\circ \tau_{k_0}$. Also,  taking the Stein factorization of  $\tau_{k_0}$, we get a birational morphism $\nu:W(k_0)\rightarrow S(k_0)$ with connected fibres and a finite morphism $\gamma:S(k_0)\rightarrow Y_0$ such that $\tau_{k_0}=\gamma\circ\nu$. 

On the one hand, since $\gamma$ is finite, by the log ramification divisor formula, we have:
\begin{equation}\label{eq23}
K_{S(k_0)}+\sum\textup{Supp}~\gamma^* G_j-R_\gamma=\gamma^*(K_{Y_0}+\sum G_j),
\end{equation}
where $R_{\gamma}$ is  effective, having no common components with $\sum\textup{Supp}\,\gamma^*G_j$. Since the right-hand side of Equation (\ref{eq23}) is $\mathbb{Q}$-Cartier, we can pull back Equation (\ref{eq23}) to $W(k_0)$ under the birational morphism $\nu$ (cf.~\cite[Lemma 2.7]{horing2014singularities}): 
\begin{align}\label{eq24}
	K_{W(k_0)}+\nu_*^{-1}(\sum_{j=1}^{s_r(1)}\textup{Supp}~\gamma^*G_j-R_{\gamma})-F=
	\tau_{k_0}^*(K_{Y_0}+\sum_{j=1}^{s_r(1)} G_j).
\end{align}
Here, $F$ is $\nu$-exceptional, which may not be effective. 
On the other hand, we recall the following equations (cf.~Equation (\ref{eq5}) and the proof of Lemma \ref{invariant1}):
\begin{align}
&K_{W(k_0)}+D(k_0)=g(k_0)^*(K_{W(k_0)}+D(k_0))+\Delta(r)_{k_0}, \label{eqq7}\\
&k_jV(r)_j=u_r^*G_j(1\le j\le s_r(1)),~~~\Delta(r)=u_r^*\Delta_{Y_0}.\label{eqe9}
\end{align}
 Putting Equation (\ref{eq24}), (\ref{eqq7}) and (\ref{eqe9})  together, we get
\begin{equation}\label{mmwmhjbjl}
	\tau_{k_0}^*(K_{Y_0}+\sum G_j)+M+F=g(k_0)^*(\tau_{k_0}^*(K_{Y_0}+\sum G_j)+M+F)+\tau_{k_0}^*\Delta_{Y_0}.
\end{equation}
Here, $M:=D(k_0)-\nu_*^{-1}(\sum_{j=1}^{s_r(1)}\textup{Supp}~\gamma^*G_j)+\nu_*^{-1}R_{\gamma}$ and we claim that $M$ is effective. Indeed, according to Equation (\ref{eqe9}) and  the notation at the beginning of Step 4, we see that,
$$\nu_*^{-1}\sum_{j=1}^{s_r(1)}\textup{Supp}~\gamma^*G_j\le \sum_{j=1}^{s_r(1)} \textup{Supp}~\tau_{k_0}^*G_j=\sum_{j=1}^{s_r(1)}R_j.$$ 
Therefore $M\ge D(k_0)-\sum R_j+\nu_*^{-1}R_\gamma\ge0$ (cf.~Claim \ref{cliam3.10}), which completes the proof of our claim. Further, applying the inductive hypothesis on $Y_0$ (cf.~Equation (\ref{eq111111})) to Equation (\ref{mmwmhjbjl}), we get
\begin{equation}\label{mmwmzybj}
M+F\equiv g(k_0)^*(M+F)+\tau_{k_0}^*\Delta_{Y_0}.
\end{equation}

Now, suppose that $\Delta_{Y_0}\neq 0$. Then $N:=\tau_{k_0}^*\Delta_{Y_0}$ is some nonzero effective $\mathbb{Q}$-divisor, which is not $\nu$-exceptional on $W(k_0)$. We shall apply the method in the proof of Lemma \ref{mmnxbxhwa} to find the contradiction. Multiplying Equation (\ref{mmwmzybj}) by a positive integer, we may assume $N$ is integral.
Then  substituting the expression of $M+F$  to the right-hand side $(m-1)$-times, we get
$$M+F=(g(k_0)^m)^* (M+F)+\sum_{j=0}^{m-1}(g(k_0)^j)^*N.$$

According to the commutative diagram, $(g(k_0)^m)^*F$ is also $\nu$-exceptional for each natural number $m$.  Fixing an ample Cartier divisor $H$ on $S(k_0)$ and pulling back to $W(k_0)$ along the birational morphism $\nu$, we get the following inequality  by projection formula:
$$(M\cdot \nu^*H^{\dim W(k_0)-1})\ge \sum_{j=0}^{m-1}\Big(\nu_*(g(k_0)^j)^*N\cdot H^{\dim S(k_0)-1}\Big).$$
Note that  $N$ is integral and not $\nu$-exceptional, and thus $\nu_*(g(k_0)^j)^*N$ is an effective integral divisor on $S(k_0)$. So the right-hand side tends to infinity if we let $m\rightarrow \infty$, a contradiction. This in turn completes the proof of Claim \ref{keyclaim} and also Theorem \ref{thm1.1} (4).
\setlength\parskip{8pt}

We give the following remark at the end of the proof of Theorem \ref{thm1.1}.\setlength\parskip{0pt}
\begin{remark}
\textup{When extending the results in \cite[Theorem 1.3]{zhang2014invariant}, we improve the original proof from the following points. First, when considering the Fano contraction $u_r:X_r\rightarrow Y_0$, the pullback of $G_j\,(j\le s_r(1))$ is a multiple of $V(r)_j$ and thus may not be reduced, since $V(r)_j$ is only $\mathbb{Q}$-Cartier; Second, when proving Proposition \ref{promm100}, $D_i$ may not be reduced while it holds for the case when $s=n+1$ (cf.~the proof of Remark \ref{mm2.24} and \cite[Proposition 2.12]{zhang2014invariant}); Finally, for each $V(r)_j$, the restriction to general fibre $W_0$ may be reducible (but still reduced) even though $V(r)_j$ is irreducible, and thus we need to                                                                                                                                                                                                                                                                                                                                                                                                                                                                                                                                                                                                                                                                                                                                                                                                                                                                                                                                                                                                                                                                                                                                                                                                                                                                                                                                                                                                                                                                                                                                                                                                                                                                                                                                                                                                                                                                                                                                                                                                                                                                                                                                                                                                                                                                                                                                                                                                                                                                                                                                                                                                                                                                                                                                                                                                                                                                                                                                                                                                                                                                                                                                                                                                                                                                                                                                                                                                                                                                                                                                                                                                                                                                                                                                                                                                                                                                                                                                                                                                                                                                                                                                                                                                                                                                                                                                                                                                                                                                                                                                                                                                                                                                                                                                                                                                                                                                                                                                                                                                                                                                                                                                                                                                                                                                                                                                                                                                                                                                                                                                                                                                                                                                                                                                                                                                                                                                                                                                                                                                                                                                                                                                                                                                                                                                                                                                                                                                                                                                                                                                                                                                                                                                                                                                                                                                                                                                                                                                                                                                                                                                                                                                                                                                                                                                                                                                                                                                                                                                                                                                                                                                                                                                                                                                                                                                                                                                                                                                                                                                                                                                                                                                                                                                                                                                                                                                                                                                                                                                                                                                                                                                                                                                                                                                                                                                                                                                                                                                                                                                                                                                                                                                                                                                                                                                                                                                                                                                                                                                                                                                                                                                                                                                                                                                                                                                                                                                                                                                                                                                                                                                                                                                                                                                                                                                                                                                                                                                                                                                                                                                                                                                                                                                                                                                                                                                                                                                                                                                                                                                                                                                                                                                                                                                                                                                                                                                                                                                                                                                                                                                                                                                                                                                                                                                                                                                                                                                                                                                                                                                                                                                                                                                                                                                                                                                                                                                                                                                                                                                                                                                                                                                                                                                                                                                                                                                                                                                                                                                                                                                                                                                                                                                                                                                                                                                                                                                                                                                                                                                                                                                                                                                                                                                                                                                                                                                                                                                                                                                                                                                                                                                                                                                                                                                                                                                                                                                                                                                                                                                                                                                                                                                                                                                                                                                                                                                                                                                                                                                                                                                                                                                                                                                                                                                                                                                                                                                                                                                                                                                                                                                                                                                                                                                                                                                                                                                                                                                                                                                                                                                                                                                                                                                                                                                                                                                                                                                                                                                                                                                                                                                                                                                                                                                                                                                                                                                                                                                                                                                                                                                                                                                                                                                                                                                                                                                                                                                                                                                                                                                                                                                                                                                                                                                                                                                                                                                                                                                                                                                                                                                                                                                                                                                                                                                                                                                                                                                                                                                                                                                                                                                                                                                                                                                                                                                                                                                                                                                                                                                                                                                                                                                                                                                                                                                                                                                                                                                                                                                                                                                                                                                                                                                                            extend the result of [Ibid.] to the case when $V_j$ are reduced (cf.~Proposition \ref{promm100}).}
\end{remark}






\end{document}